\documentclass{amsproc}
\usepackage{euscript, graphicx, epstopdf}
\usepackage{cases}
\usepackage{mathrsfs}
\usepackage{bbm}
\usepackage{amssymb}
\usepackage{txfonts}
\usepackage{amsfonts,latexsym,amsmath,amsxtra,mathdots,amssymb,latexsym,mathabx}
\usepackage[all,cmtip]{xy}
\usepackage{color}
\usepackage{multicol}
\usepackage{hyperref}
\usepackage{tikz}
\usepackage{float}
\usepackage{setspace, floatflt}
\usepackage{url}

\allowdisplaybreaks

\def\mod{\mathrm{mod}\ }

\newcommand{\BA}{{\mathbb {A}}} 
\newcommand{\BC}{{\mathbb {C}}}

 \newcommand{\BL}{{\mathbb {L}}}
 \newcommand{\BN}{{\mathbb {N}}}
 
\newcommand{\BQ}{{\mathbb {Q}}} \newcommand{\BR}{{\mathbb {R}}}
\newcommand{\BS}{{\mathbb {S}}}

 \newcommand{\BZ}{{\mathbb {Z}}}

\newcommand{\GL}{{\mathrm {GL}}} 
\newcommand{\SL}{{\mathrm {SL}}} 

 \newcommand{\Tr}{{\mathrm{Tr}}}

\newcommand{\RN}{{\mathrm {N}}}

\newcommand{\sstyle}{\scriptstyle}

\newcommand{\ra}{\rightarrow}

\def\fra{\mathfrak{a}}

\def\frb{\mathfrak{b}}

\def\frO{\text{$\text{\usefont{U}{BOONDOX-cal}{m}{n}O}$}\hskip 1pt}
\def\frOO{\text{$\text{\usefont{U}{BOONDOX-cal}{m}{n}O}$}}
\def\frp{\mathfrak{p}}

\def\-{^{-1}}

\def\udelta{\boldsymbol {\delta}}
\def\ulambda{\boldsymbol \lambda}

\def\umu{\boldsymbol \mu}

\def\ue{\boldsymbol e}
\def\um{\boldsymbol m}

\def\lp {\left (}
\def\rp {\right )}
\def\EC{\EuScript C}

\def\BCx{\BC^\times}
\def\Voronoi{Vorono\" \i \hskip 3 pt}

\renewcommand{\Re}{{\mathrm{Re}\,}}

\def\Fx{F^{\times}}
\def\vv{\varv}

\def\nwedge {\hskip - 2 pt \wedge \hskip - 2 pt }

\makeatletter
\g@addto@macro\normalsize{\setlength\abovedisplayskip{3pt}}
\makeatother

\makeatletter
\g@addto@macro\normalsize{\setlength\belowdisplayskip{3pt}}
\makeatother

\newcommand{\delete}[1]{}

\theoremstyle{plain}

\newtheorem{thm}{Theorem}[section] \newtheorem{cor}[thm]{Corollary}
\newtheorem{lem}[thm]{Lemma}  \newtheorem{prop}[thm]{Proposition}

\newtheorem {rem}[thm]{Remark}

\newtheorem*{acknowledgement}{Acknowledgements}

\numberwithin{equation}{section}

\begin{document}

	\title[Cancellation in the Additive Twists of Fourier Coefficients]{Cancellation in the Additive Twists of Fourier Coefficients for $\GL_2$ and $\GL_3$ over Number Fields}
	
	\author{Zhi Qi}
	\address{School of Mathematical Sciences\\ Zhejiang University\\Hangzhou, 310027\\China}
	\email{zhi.qi@zju.edu.cn}
	

	\begin{abstract}
		In this article, we study the sum of additively twisted Fourier coefficients of an irreducible cuspidal automorphic representation of $\GL_2$ or $\GL_3$  over an arbitrary number field. When the representation is unramified at all non-archimedean places, we prove the Wilton type bound for $\GL_2$ and the Miller type bound for $\GL_3$ which are uniform in terms of the additive character.
	\end{abstract}
	
	\maketitle

    \section{Introduction}

    \subsection{Backgrounds}

    It is a classical problem to estimate sums involving the Fourier coefficients of a modular form. For example, when $f(z)$ is a holomorphic cusp form of weight $k$ for $\SL_2 (\BZ)$ with the Fourier expansion,
    $$f (z) = \sum_{n=1}^{\infty} A_f (n) n^{(k-1)/2} e (n z),$$
    with $e(z) = e^{2 \pi i z}$, it is well known that for any real $\theta$ and $T \geqslant 1$,
    \begin{align}\label{0eq: holomorphic form}
    \sum_{n \leqslant T} A_f (n) e (\theta n) \lll_f T^{1/2} \log (2 T),
    \end{align}
    with the implied constant depending only on $f$ (see \cite[Theorem 5.3]{Iwaniec-Topics}). This is a classical estimate due to Wilton. Because of the square-root cancellation, which is essentially best possible, \eqref{0eq: holomorphic form} shows that there is no correlation between the Fourier coefficients $A_f (n) $ and   additive characters $ e (\theta n) $. In the work  of Stephen D. Miller  \cite{Miller-Wilton},  for  the Fourier coefficients $A_f (q, n)$ of a cusp form $f$ for $\GL_3 (\BZ)$   he proved
    \begin{align}\label{0eq: GL3 estimate}
    \sum_{n  \leqslant T} A_f (q, n) e (\theta n) \lll_{\, \varepsilon, q, f} T^{3/4 + \varepsilon},
    \end{align}
    with  the implied constant depending only on $\varepsilon$, $q$ and the cusp form $f$. This is halfway between the trivial bound of $O(T)$ and the best-possible bound of $O \big(T^{1/2} \big)$. Uniformity in $\theta$ is an important feature of these estimates of Wilton and Miller.
    
    The estimate in \eqref{0eq: holomorphic form} can be seen in various ways, but the easiest is perhaps from the $\SL_2 (\BZ)$-invariance of $|f (z)| (\mathrm{Im }\, z)^{ k/2} $. Indeed, from the modularity we have $|f(z)| = O_f \lp (\mathrm{Im }\, z)^{- k/2} \rp$ for all $z = \theta + i/T$ in the upper half plane, and in particular for any real $\theta$
    \begin{align}\label{0eq: smooth, holomorphic}
    \sum_{n=1}^{\infty} A_f (n)  e (n \theta) n^{(k-1)/2} \exp (-2\pi n/T) \lll_f T^{k/2}.
    \end{align}
    This is in essence a smoothed form of \eqref{0eq: holomorphic form}. As observed by Miller, finding a matrix in $\SL_2 (\BZ)$ that lifts $ z = \theta + i/T $ to a fixed fundamental domain is a Diophantine problem, while the modularity can be encoded in a Vorono\"i-type summation formula. With these observations, Miller deduced his estimate in \eqref{0eq: GL3 estimate}  from   Diophantine approximations of $\theta$ and the $\GL_3$ \Voronoi summation formula in \cite{Miller-Schmid-2006} which involves the Fourier coefficients $A_f (q, n)$. 
    
    We now make some remarks on the estimates similar to \eqref{0eq: holomorphic form} for Maa\ss \hskip 2 pt cusp forms for $\GL_2$. When $f (z)$ is a Maa\ss \hskip 2pt cusp form with Laplacian eigenvalue $\lambda_f$,  the estimate in \eqref{0eq: holomorphic form} holds with a hybrid bound $ (\lambda_f T)^{1/2} \log (2 T) $ which is also uniform in $f$ (see \cite[Theorem 8.1]{Iw-Spectral} and its proof is similar to that of \cite[Theorem 5.3]{Iwaniec-Topics}). While Miller did not treat $\GL_2$ Maa\ss \hskip 2 pt forms in \cite{Miller-Wilton}, his ideas as described above was used later in \cite{Additive-Godber} to get an even better bound $ \lambda_f^{1/4 + \varepsilon} T ^{1/2 + \varepsilon} $.
    
     \vskip 5 pt

    The purpose of this paper is to obtain  estimates analogous to \eqref{0eq: holomorphic form} and \eqref{0eq: GL3 estimate}   for the  Fourier coefficients of a  cuspidal automorphic representation of $\GL_2$ or $\GL_3$ over an arbitrary number field. As an illustration, let $F$ be a totally real field of degree $N$ and consider a Hilbert modular form $f (z)$    of weight $ k $ for the Hilbert modular group $\SL_2 (\frO)$ with the Fourier expansion,
    \begin{align*}
    f (z) = \sum_{n \, \in \frOO^{\, \prime}}  A_f (n) \RN \big( n^{ (k-1)/2}\big) e (\Tr (n z)), \hskip 10 pt \RN \big( n^{ (k-1)/2}\big) = \prod_{j=1}^N n_j^{(k_j -1)/2}, \Tr (n z) = \sum_{j=1}^N n_j z_j.
    \end{align*}
    Here $\frO$ is the ring of integers in $F$ and $\frO'$ is its dual. Then the function $|f(z)| \RN  \big( \mathrm {Im}\, z ^{k/2} \big)$ % with the usual notation $\RN (\mathrm {Im}\, z)^{k/2} = \prod_{j=1}^N \mathrm {Im}\, z_j^{k_j/2}$, 
    is $\SL_2 (\frO)$-invariant, which, by setting $z_j = \theta_j + i/T$, leads to an estimate analogous to \eqref{0eq: smooth, holomorphic}, 
    \begin{align}
    \sum_{n \, \in \frOO^{\, \prime} } A_f (n)  e (\Tr(n \theta)) \RN \big( n^{ (k-1)/2}\big) \exp (-2\pi \, \Tr (n) /T) \lll_f T^{\Tr (k)/2},
    \end{align} 
    for any $\theta \in \BR^N$. In order to generalize this further, we shall use the \Voronoi summation formula over number fields in the paper of Ichino and Templier  \cite{Ichino-Templier}. To estimate the Hankel transforms in the Vorono\"i summation, we shall apply the asymptotic formulae for Bessel kernels over both real and complex numbers in the author's recent work \cite{Qi-Bessel}.

    \subsection{Statement of Results}

     Let $F$ be a number field and let $\frO$ be the ring of integers and $\frO'$ be its dual lattice. Let $\mathfrak{D} $ denote the different ideal so that $ \frO' = \mathfrak{D}^{-1} $. Let $N$ be the degree of $F$. Let $S_{\infty}$ be the set of all archimedean places of $F$. For $\vv \in S_{\infty}$, let $F_{\vv}$ be the corresponding local field. % with the usual absolute value $|\ |_{\vv}$. 
    Define $F_{\infty} = \prod_{ \vv\, \in S_{\infty}} F_{\vv} = \BR^N$ and the trace on $F_{\infty}$ by $\Tr \, x = \sum_{ \vv\, \in S_{\infty}} \Tr_{F_{\vv}/ \BR} \, x_{\vv}$.  Let $\varPi'  \subset F_{\infty}$ be a fixed fundamental parallelotope for  $F_{\infty} / \frO'$ which is symmetric about zero.
    
    Let $\pi$ be   an irreducible cuspidal automorphic representation of $\GL_r$ over $F$, with $r = 2, 3$. Suppose that $\pi$ is unramified  and has trivial central character at all non-archimedean places of $F$.  For  nonzero ideals $\fra \subset \frO  $, let $A_{\pi} (\fra)$ denote the Fourier coefficients of $\pi$ (when $r = 3$,  in order to unify the presentation, we let $A_{\pi} (\fra)$ be the shorthand notation for $A_{\pi} (\mathfrak{D}, \fra)$). 
    
    In this paper, we study the sums of Fourier coefficients twisted by additive characters
\begin{align}\label{0eq: S theta (T)}
S_{\theta} (T) = \sum_{\sstyle \gamma\, \in \frOO \hskip 0.5 pt ' \smallsetminus \{0\} \cap T \cdot \varPi' } A_{\pi} (\gamma \mathfrak{D}) e \lp \Tr (\theta \gamma) \rp , \hskip 10 pt \theta \in F_{\infty} ,
\end{align}
and we shall be concerned with obtaining estimates for $S_{\theta} (T) $ which are uniform in $\theta$.
%with $\BR^N (T) = \left\{ x \in F_{\infty} = \BR^N : |x_n | \leqslant T \text{ for all } n=1,..., N \right\}$.
%Our main results are the nontrivial uniform estimates for $S_{\theta} (T)$ of $O \lp (\log T)^N T^{N/2} \rp$ for  $\GL_2$ Fourier coefficients  and  $O_{\varepsilon} \lp  T^{3N/4 + \epsilon} \rp$ for $\GL_3$.

Thanks to the Rankin-Selberg theory (\cite{J-S-Rankin-Selberg}), we know that $A_{\pi}(\gamma)$  obey the Ramanujan conjecture on average. More precisely, we have 
\begin{align}\label{0eq: 2nd moment}
\sum_{\sstyle \gamma\, \in \frOO \hskip 0.5 pt ' \smallsetminus \{0\}  \cap T \cdot \varPi'  } |A_{\pi} (\gamma \mathfrak{D})|^2  
= O \big( T^{N+\varepsilon} \big),
\end{align}
(the $\varepsilon$ in the exponent is due to the possible infinitude of the group of units in $\frO$ and $T^{\varepsilon}$ may be reduced to a power of $\log T$.), and by Cauchy-Schwarz,
\begin{align}
\sum_{\sstyle \gamma\, \in \frOO \hskip 0.5 pt ' \smallsetminus \{0\}  \cap T \cdot \varPi'  } |A_{\pi} (\gamma \mathfrak{D})|   
= O \big( T^{N+\varepsilon} \big).
\end{align}
See \S \ref{sec: average}. Thus the trivial estimate for $S_{\theta} (T)$  is $O \lp T^{N + \varepsilon} \rp$, obtained by taking the absolute value of each term in  \eqref{0eq: S theta (T)}.  On the other hand, by Parseval's identity,
\begin{align*}
\int_0^1 |S_{\theta} (T) |^2 d \theta = \sum_{\sstyle \gamma\, \in \frOO \hskip 0.5 pt ' \smallsetminus \{0\}  \cap T \cdot \varPi'  } |A_{\pi} (\gamma \mathfrak{D})|^2,
\end{align*}
Hence, in view of \eqref{0eq: 2nd moment}, the   best possible uniform estimate for $ S_{\theta} (T)  $ would be $O \lp T^{N/2 + \varepsilon} \rp$.
Our main results are the following uniform estimates for $S_{\theta} (T) $ analogous to \eqref{0eq: holomorphic form} and \eqref{0eq: GL3 estimate}. In the $\GL_2$ case, we have the best-possible bound of Wilton type, whereas, in the $\GL_3$ case, we have Miller's bound, which is halfway between the trivial bound and the best-possible bound.

\begin{thm}\label{thm: main}
	Let notations be as above.
	Suppose that $T > 0$ is sufficiently large in terms of $\pi$ and $F$.  
	Then, for any $\theta \in F_{\infty}$, we have
	\begin{align}\label{0eq: sum, r= 2}
	\sum_{\sstyle \gamma\, \in \frOO \hskip 0.5 pt ' \smallsetminus \{0\}  \cap T \cdot \varPi'  } 
	A_{\pi} (\gamma \mathfrak{D}) e \lp \Tr (\theta \gamma) \rp  
	= O \big( T^{N/2} (\log T)^N  \big),
	\end{align}
	if $r = 2$, and for any $\varepsilon > 0$,
	\begin{align}\label{0eq: sum, r= 3}
	\sum_{\sstyle \gamma\, \in \frOO \hskip 0.5 pt ' \smallsetminus \{0\}  \cap T \cdot \varPi'   } 
	A_{\pi} (\mathfrak{D}, \gamma \mathfrak{D}) e \lp \Tr (\theta \gamma) \rp   
	= O_{\varepsilon} \big(  T^{3 N /4 + \varepsilon} \big),
	\end{align}
	if $r = 3$, where the implied constants in {\rm (\ref{0eq: sum, r= 2}, \ref{0eq: sum, r= 3})} depend only on  $\pi$, $F$, the choice of $\varPi$, and in addition on $\varepsilon$ for \eqref{0eq: sum, r= 3}.
\end{thm}

We shall deduce Theorem \ref{thm: main} from its smoothed version as follows.

\begin{thm}\label{thm: main, smooth}
	Let notations be as above. Let $ \varDelta > 1$.
	Let $T > 0$ be sufficiently large in terms of $\pi$ and $F$. For each archimedean place $\vv $, let $w_{\vv}$ be a smooth function on $F_{\vv}$ such that 
	\begin{itemize}
		\item [-] for $F_{\vv} = \BR$, $w_{\vv} (x)$ is supported on $\{ x \in \BR : |x| \in [T, \varDelta T] \}$  and
		\begin{align}\label{0eq: weight, real}
		(d/dx)^j w_{\vv} (x) \lll_{j, \, \varDelta} T^{- j},
		\end{align} 
		\item [-] for $F_{\vv} = \BC$,  $w_{\vv} (z)$ is   supported on $\{ z \in \BC : |z| \in [T, \varDelta T] \}$ and  
		\begin{align}\label{0eq: weight, complex}
		(\partial/\partial z)^{j} (\partial/\partial \overline z)^{j'} w_{\vv} (z) \lll_{j,   j',  \, \varDelta} T^{ - j - j' },
		\end{align}
		or equivalently, in the polar coordinates,
		\begin{align}\label{0eq: weight, complex, 2}
		(\partial/\partial x)^{j} (\partial/\partial \phi)^{k} w_{\vv} \lp x e^{i \phi} \rp \lll_{j, \,  k,  \, \varDelta} T^{ - j  }.
		\end{align}
	\end{itemize}
	Let $w$ be the product function of $w_{\vv}$ on $F_{\infty}$. 
	Then, for any  $\theta \in F_{\infty}$, we have
	\begin{align}\label{0eq: weighted sum, r= 2}
	\sum_{ \gamma \, \in \frOO \hskip 0.5 pt '  \smallsetminus \{0\}    } A_{\pi} (\gamma \mathfrak{D}) e \lp \Tr (\theta \gamma) \rp w (\gamma) = O \big(  T^{N/2} \big),
	\end{align}
	if $r = 2$, and for any $\varepsilon > 0$,
	\begin{align}\label{0eq: weighted sum, r= 3}
	\sum_{ \gamma \, \in \frOO \hskip 0.5 pt '  \smallsetminus \{0\}   } A_{\pi} (\mathfrak{D}, \gamma \mathfrak{D}) e \lp \Tr (\theta \gamma) \rp w (\gamma) = O_{\varepsilon} \big(  T^{3 N /4 + \varepsilon} \big),
	\end{align}
	if $r = 3$, where the implied constants in {\rm (\ref{0eq: weighted sum, r= 2}, \ref{0eq: weighted sum, r= 3})} depend only on  $\pi$, $F$, $\varDelta$ and those  in {\rm(\ref{0eq: weight, real}, \ref{0eq: weight, complex}, \ref{0eq: weight, complex, 2})}, and in addition on $\varepsilon$ for \eqref{0eq: weighted sum, r= 3}.
\end{thm}

\subsection{Remarks} \

In Miller's paper \cite{Miller-Wilton}, Hankel transforms in the Mellin-Barnes   form are analyzed using  Stirling's asymptotic formula for the Gamma function. This approach is very restrictive from the analytic aspect in the complex case, as it can only be applied to spherical weight functions. Nevertheless, the spherical assumption is usually sufficient, for example, for deducing Theorem \ref{thm: main} from \ref{thm: main, smooth}. In contrast, we use the asymptotic formulae for Bessel kernels, which turn Hankel transforms into oscillatory integrals so that we may directly use the method of stationary phase. Our method does not only yield better estimates for Hankel transforms but also enables us to treat complex Hankel transforms of the large class of compactly supported weight functions described in Theorem \ref{thm: main, smooth}.  See \S \ref{sec: asymptotic} for more details.

In  this paper, same as \cite{Miller-Wilton}, the cuspidal automorphic representation $\pi$ is considered fixed.  There however have been papers after \cite{Miller-Wilton},   for example \cite{Additive-LiY, Additive-Godber, Additive-Li}, on obtaining estimates which are also uniform in the archimedean parameters of $\pi$. These bounds are useful when one considers  varying automorphic forms, for example, in the subconvexity problem (see \cite{Ho-Mu-Qi}). With Ichino and Templier's \Voronoi summation formula, the author believes that they may also be generalized to an arbitrary number field.  When the parameters are taken into account, one has to use the Mellin-Barnes form of Hankel transforms and Stirling's asymptotic formula, or one needs to get an asymptotic formula for Bessel kernels at the transition range, which is currently only known for $\GL_2(\BR)$.

In higher rank $r \geqslant 4$, with more efforts, Ichino and Templier's \Voronoi summation formula for $\GL_r$ may   be reformulated in the classical language as in \S \ref{sec: Voronoi}, with hyper-Kloosterman sums on the right hand side of the identity. Also, the analysis of Hankel transforms in \S \ref{sec: asymptotic} may   be generalized for $r \geqslant 4$.  However,  the application of the $\GL_r$ \Voronoi summation formula with Weil's bound on  each individual hyper-Kloosterman sum fails to  yield nontrivial bound when $r \geqslant 4$.  %To the author's knowledge, there were no application of the  \Voronoi formula for $\GL_4 $ until the very recent work of Blomer, Li and Miller \cite{BLM-GL4} on the non-vanishing problem for the $L$-functions for $\GL_4 \times \GL_2$.  An important ingredient is a {\it balanced}  \Voronoi formula for $\GL_4$   involving Kloosterman sums on both sides (\cite{MZ-Voronoi}).

Finally, we remark that in our theorems the sums of Fourier coefficients are restricted to the ideal class containing $ \mathfrak{D}$. Our results should be generalized by using Theorem 3 of Ichino and Templier \cite{Ichino-Templier} (instead of their Theorem 1) along with the Hecke relation.

\begin{acknowledgement}
	The author would like to thank  Stephen D. Miller for valuable comments and helpful discussions. The author wishes to thank the   referee for thorough reading of the manuscript and several suggestions which helped improving the exposition.
\end{acknowledgement}

    \section{Notations for Number Fields}
    
%\subsection{Notations for Number Fields}

We now give a list of our most frequently used notations from algebraic number theory. See \cite{Lang-ANT} for more details.
    
Let $F$ be a number field of degree $N$. Let $\frO$ be its ring of integers and $\frOO^{\times}$ the group of units. Let $\frO '$ be the dual of $\frO$.  Let $\mathfrak{D}$ be the different ideal of $F$, namely, $\mathfrak{D}^{-1} = \frO '$.
Let $\mathrm{N}$ and $\Tr$ denote the norm and the trace for $F$, respectively. Denote by $ \BA = \BA_F$ the  adele ring of $F$.

For any place $\vv$ of $F$, we denote by $F_{\vv}$ the corresponding local field. 
When $\vv$ is non-archimedean,  let $\frp_{\vv}$ be the corresponding prime ideal of   $\frO $, let $\mathrm{ord}_{\vv}$ denote the additive valuation.
Let  $N_{\vv}$ be the local degree of $F_{\vv}$; in particular, $N_{\vv} = 1$ if $F_{\vv} = \BR$ and  $N_{\vv} = 2$ if $F_{\vv} = \BC$.  Let $\|\, \|_{\vv}$  denote the normalized module of $F_{\vv}$. We have $\|\,\|_{\vv} = |\ |$ if $F_{\vv} = \BR$ and  $\|\,\|_{\vv} = |\ |^2$ if $F_{\vv} = \BC$, where $|\ |$ is the usual absolute value.

Fix the (non-trivial) standard additive character $\psi = \otimes_{\vv} \psi_{\vv}$ on $\BA/F$ as in \cite[\S XIV.1]{Lang-ANT} such that $\psi_{\vv} (x) = e (- x)$ if $F_\vv = \BR$,   $\psi_{\vv} (z) = e (- (z + \overline z))$ if $F_\vv = \BC$, %where $e (x) = e^{2 \pi i x}$, 
and that $\psi_{\vv}$ has conductor $\mathfrak{D}_{\vv}\-$ for any non-archimedean $F_\vv$ ($\mathfrak{D}_{\varv}$ is the local different ideal). Let $S_{\infty}$, respectively $S_f$, denote the set of all archimedean, respectively non-archimedean,  places of $F$. Accordingly, we split %$\|\,\| = \|\,\|_f \, \|\,\|_{\infty}$ and 
$\psi = \psi_f \psi_{\infty} $.
Note that %$\|\gamma\|_f = |\RN \gamma|$ and 
$\psi_{\infty} (\gamma) = e (- \Tr \, \gamma)$ and hence $\psi_{f} (\gamma) = e (  \Tr \, \gamma)$ for $\gamma \in F$. 

We choose the Haar measure of $F_{\vv}$ as in  \cite[\S XIV.1]{Lang-ANT}; the Haar measure is the ordinary Lebesgue measure on the real line if $F_{\vv} = \BR$ and twice  the ordinary Lebesgue measure on the complex plane if $F_{\vv} = \BC$. 

For each $\vv \in S_{\infty}$, let $\sigma_{\vv}$ denote the embedding $F \hookrightarrow F_{\vv}$, %Let $r_1$ and $r_2$ be the number of real and complex embeddings, respectively. 
then all the  $\sigma_{\vv}$ yield an embedding $\sigma : F \hookrightarrow F_{ {\infty}} = \prod _{\vv \, \in S_{\infty}} F_{\vv} = \BR^N$.  % See \cite{Lang-ANT} for more details.
Let $r_1$  be the number of real and $r_2$ be the number of conjugate pairs of complex embeddings of $F$. 

%\subsection{Space of Dyadic Weight Functions}

	\section{\texorpdfstring{\Voronoi Summation Formulae for $\GL_2$ and $\GL_3$ over Number Fields}{Voronoi Summation Formulae for $\GL_2$ and $\GL_3$ over Number Fields}}\label{sec: Voronoi}
	
	Our main tool is the $\GL_2$ and $\GL_3$ \Voronoi summation formulae over number fields in the work of Ichino and Templier  \cite{Ichino-Templier}. It is formulated in an adelic framework. We shall give a brief summary of the adelic formulae and then translate them into the  classical language.

	\subsection{Adelic \Voronoi Summation Formulae  \cite{Ichino-Templier}} \label{sec: adelic Voronoi}
	We shall follow the notations  in \cite{Ichino-Templier}, and for further details readers can refer to \cite{Ichino-Templier} and the works cited there. 
	
	Let $\pi = \otimes_{\varv} \pi_{\varv}$ be   an irreducible cuspidal automorphic representation of $\GL_r(\BA)$, with $r = 2, 3$.  Let $\widetilde {\pi} = \otimes_{\varv} \widetilde \pi_{\varv}$ be the contragradient  representation of $\pi$.
	
	Let $S$ be a finite set of places of $F$ including the ramified places of $\pi$. %and $\psi$  and all archimedean places. 
	Denote by $\BA^S$ the subring of adeles
	with trivial component above $S$. Denote by $W^S_{\text{o}} = \prod_{\varv \, \notin S} W_{\text{o}  \, \varv}$ the unramified Whittaker
	function of $\pi^S = \otimes_{\varv \, \notin S} \pi_{\varv}$ above the complement of $S$. Let  $\widetilde W^S_{\mathrm{o}} = \prod_{\varv \, \notin S} \widetilde W_{\mathrm{o}  \, \varv}$ be the unramified Whittaker function of   $\widetilde \pi^S = \otimes_{\varv \, \notin S} \widetilde \pi_{\varv}$. %We have $ \widetilde W^S_{\mathrm{o}} (g) =   W^S_{\mathrm{o}} \lp \varw_r {^t  g \-}   \rp $, $g \in \GL_r (\BA^S)$, where $\varw_r$ is the long Weyl element of $\GL_r$,
	\delete{ \begin{align*}
	\varw_2 = \begin{pmatrix}
	 &  1 \\
	1 &
	\end{pmatrix}, \hskip 10 pt 
	\varw_3 = \begin{pmatrix}
	& &  1 \\
	& 1 & \\
	1 & & 
	\end{pmatrix}.
	\end{align*}
}
	    
	    For any place $\varv$ of $F$, to a smooth compactly supported function $w_{\varv} \in C_c^{\infty} (F_{\varv}^{\times})$ is associated a dual function $\widetilde w_{\varv}$  of $ w_{\varv} $ such that
	    \begin{equation}\label{2eq: Ichino-Templier}
	    \begin{split}
	    \int_{ F_{\varv}^\times} \widetilde w_{\varv} (x) \chiup (x)\- & \|x\|_{\varv}^{s - \frac {r-1} 2} d^\times x \\
	    & = \chiup (-1)^{r-1} \gamma (1-s, \pi_{\varv} \otimes \chiup, \psi_{\varv} ) \int_{ F_{\varv}^\times} w_{\varv} (x) \chiup (x) \|x\|_{\varv}^{1 - s - \frac {r-1} 2} d^\times x,
	    \end{split}
	    \end{equation}
	    for all $s$ of real part sufficiently large and all unitary multiplicative characters $\chiup $ of $ F_{\varv}^\times$.
	    The equality \eqref{2eq: Ichino-Templier} %is independent of the chosen Haar measure $d^\times x$ on $ F_{\varv}^\times$ and 
	    defines  $\widetilde w_{\varv}$ uniquely in terms of $\pi_{\varv}$, $\psi_{\varv}$ and $w_{\varv}$. For $S$ as above, we put $
	    w_S = \prod_{ \varv \, \in S }  w_{\varv}  $, $ \widetilde w_S = \prod_{ \varv \, \in S } \widetilde w_{\varv}  $.
	    
	    Let $\varv$ be an unramified place of $\pi$. It should be kept in mind that, since the additive character $\psi_{ \vv }$ has conductor $\mathfrak{D}_{\vv}\-$,  $W_{\mathrm{o}\, \vv} \begin{pmatrix}
	    \gamma & \\
	    & 1 
	    \end{pmatrix}$ or $W_{\mathrm{o}\, \vv} \begin{pmatrix}
	    \gamma_1 \gamma_2 & & \\
	    & \gamma_1 & \\
	    & & 1 
	    \end{pmatrix}$ vanishes unless $ \gamma \in \mathfrak{D}_{\vv}\-$ or both $\gamma_1, \gamma_2 \in \mathfrak{D}_{\vv}\-$. Let $ \delta_{ \vv}  $ be a generator of $\mathfrak{D}_{\vv}$. For $\gamma, \zeta \in \Fx_{\varv}$, we define
	    $K_{\varv} (\gamma, \zeta, \widetilde W_{\mathrm{o} \, \varv}) $ as follows. When $r = 2$,  we % set   $K_{\varv} (\gamma, \zeta, \widetilde W_{\mathrm{o} \, \varv}) = 0$ if  $\|\delta_{ \vv} \gamma  \|_{\varv} > \|  \zeta \|_{\varv}^2$, and 
	    let
	    \begin{equation}\label{2eq: K, r=2}
	    K_{\varv} (\gamma, \zeta, \widetilde W_{\mathrm{o} \, \varv}) = 
	    \widetilde W_{\mathrm{o} \, \varv} \begin{pmatrix}
	      \gamma \zeta\- & \\
	    & \zeta
	    \end{pmatrix} \psi_{\varv} \lp \gamma \zeta\- \rp.
	    \end{equation}
	    %if $\|\delta_{ \vv} \gamma  \|_{\varv} \leqslant \|  \zeta \|_{\varv}^2$. 
	    When $r = 3$, %for $\|\delta_{ \vv}^3 \gamma  \|_{\varv} > \|  \zeta \|_{\varv}^3  $ we again set $ K_{\varv} (\gamma, \zeta, \widetilde W_{\mathrm{o} \, \varv}) = 0 $, and for $\|\gamma  \|_{\varv} \leqslant \|  \zeta \|_{\varv}^3$ 
	    we let
	    \begin{equation}\label{2eq: K, r=3}
	    \begin{split}
	    K_{\varv} (\gamma, \zeta, \widetilde W_{\mathrm{o} \, \varv}) =   \|\zeta\|_\vv   \sum_{\sstyle \gamma' \, \in \Fx_\varv / \frOO_{\,\varv}^{\times} \atop { \sstyle  1 \leqslant \|\gamma'\|_{\varv} \leqslant \| \delta_{ \vv} \- \zeta\|_{\varv} \atop \sstyle \| \delta_{ \vv}  \gamma \zeta\-\|_{\varv} \leqslant \|\gamma'\|_{\varv}^2  } } \widetilde W_{\mathrm{o} \, \varv} 
	    \begin{pmatrix}
	       \gamma \gamma'^{\, -1} \zeta\-  & & \\
	    &  \gamma' & \\
	    & &   \zeta
	    \end{pmatrix} Kl_\varv (  \gamma \zeta\-,   \gamma'), 
	    \end{split}
	    \end{equation}
	in which  $Kl_\varv ( \gamma \zeta\-,   \gamma')$ is the Kloosterman sum defined by
	    \begin{align}\label{2eq: Kloosterman, r=3}
	    Kl_\varv ( \gamma \zeta\-, \gamma') =  \sum_{ \nu \, \in \gamma' \frOO_{\,\varv}^\times / \frO_\varv } 
	    \psi_\varv \lp   \nu +   \gamma  \zeta\- \nu\- \rp.
	    \end{align} 
	      %Here, $F_R = \prod_{\varv \in R} F_{\varv}$, $ \frO_R = \prod_{\varv \in R} \frO_{\varv} $ and  $ \frO_R^{\times} = \prod_{\varv \in R} \frO_{\varv}^{\times} $, $ \widetilde W_{\mathrm{o} \, R} = \prod_{ \varv \in R} \widetilde W_{\mathrm{o} \, \varv}$. 
	    In the quotient $ \gamma' \frOO^\times_{\, \varv} / \frO_\varv $ above, the group $ \frO_{\varv}$ acts additively on $\gamma'_{\varv} \frOO^\times_{\,\varv}$ if $\|\gamma'\|_{\varv} > 1$ and $ \gamma'_{\varv} \frOO^\times_{\,\varv} / \frO_\varv = \{1\}$ if $\|\gamma'\|_{\varv} = 1$, that is, imposing that $\nu  =  1$. 
	    Now, let $R$ be a finite set of places where $\pi$ is unramified. We define $   \widetilde W_{\mathrm{o} \, R} = \prod_{\varv \, \in R} \widetilde W_{\mathrm{o} \, \varv} $ and $ K_{R} (\gamma, \zeta, \widetilde W_{\mathrm{o} \, R}) = \prod_{\varv \, \in R} K_{\varv} (\gamma_{\varv}, \zeta_{\varv}, \widetilde W_{\mathrm{o} \, \varv}) $ for $\gamma, \zeta \in \Fx_R = \prod_{\varv \, \in R} \Fx_{\varv}$.

	\begin{prop} \label{prop: Voronoi, adelic} {\rm (\cite[Theorem 1]{Ichino-Templier}).} Let $\zeta \in \BA^S$, let $ R $ be the set of places $\varv$ such that $\|\zeta\|_{\varv} > 1$, and for all $\varv \in S$ let $w_{\varv} \in C_c^{\infty} (F_{\varv}^{\times})$. Then we have the identity
		\begin{align}\label{2eq: Voronoi, adelic}
	 \sum_{\gamma \, \in \Fx}  \hskip -2 pt \psi (\gamma \zeta) W_{\mathrm{o}}^S   
	 \begin{pmatrix}
	 \gamma & \\
	 & 1_{r-1}
	 \end{pmatrix}
	  w_{S} (\gamma)
	 = \sum_{\gamma \, \in \Fx} \hskip -2 pt K_{R} (\gamma, \zeta, \widetilde W_{\mathrm{o} \, R})  \widetilde W_{\mathrm{o}}^{S \cup R}   
	 \begin{pmatrix}
	 \gamma & \\
	 & 1_{r-1}
	 \end{pmatrix}
	 \widetilde  w_S (\gamma).
		\end{align}
	\end{prop}

\begin{rem}
	We remark that the sign in front of $\gamma  \zeta^{-1} \nu^{-1} $ in \eqref{2eq: Kloosterman, r=3} is negative in  \cite{Ichino-Templier}, indeed $(-)^{r}$, but it is always positive in \cite{Miller-Schmid-2006,Miller-Schmid-2009}. The reason for this inconsistency is not clear to the author, but he believes that the   sigh choice in the latter should be correct. While the sigh does not matter for our problem,  it could be highly sensitive in the subconvexity problems. 
\end{rem}

\begin{rem}
	It is assumed in Theorem {\rm 1} of {\rm\cite{Ichino-Templier}} that $S$   contains all the ramified places of the  additive character  $\psi$. By checking their proof, we find that this assumption is rather superfluous and may be removed. 
	For $r = 2$, the formula {\rm\eqref{2eq: K, r=2}} is solid without any assumptions on $\psi_{\vv}$. For $r = 3$, their arguments may be generalized so that $\psi_{\vv}$ is only required to be trivial on $\frO_{\vv}$. 

\end{rem}
	
	 \subsection{\Voronoi Summation in Classical Formulation}

		Subsequently, we assume that $\pi$ is unramified  and has trivial central character at all the non-archimedean places. 
	 Then one may choose $S = S_{\infty}$. In practice, one usually let    $\zeta \in \BA^{S_{\infty}}$   be the diagonal embedding of a fraction $\alpha / \beta \in F$, and it is preferable to have a   classical formulation of the \Voronoi summation in terms of Fourier coefficients, exponential factors, Kloosterman sums and Hankel transforms.  However, when $F$ does not have class number one,  there is some subtlety   due to the fact that the fraction $\alpha / \beta$ can not alway be chosen so that $\alpha$ and $\beta$ are relatively prime integers.  Therefore the look of our translation of Ichino and Templier's \Voronoi summation formula is slightly different from those in \cite{Miller-Schmid-2004-1, Miller-Schmid-2006} (see Remark \ref{rem: Voronoi}).
	 
	 \vskip 5 pt
	
	When $r = 2$, for  a nonzero ideal $\fra $ of $\frO$ we define the Fourier coefficient 
	\begin{align}\label{2eq: Fourier coefficients, r=2}
	A_{\pi} (\fra) =    \mathrm{N} \lp \fra \mathfrak{D}\- \rp^{    1 / 2} W_{\mathrm{o}}^{S_{\infty}}  
	\begin{pmatrix}
	\fra \mathfrak{D}\- & \\
	& 1 
	\end{pmatrix}.
	\end{align} 
	When $r = 3$, for  nonzero ideals $\fra , \fra '  $ we define  the Fourier coefficient
	\begin{align}\label{2eq: Fourier coefficients, r=3}
	A_{\pi} (\fra' , \fra ) =    \mathrm{N} \lp \fra  \fra' \mathfrak{D}^{-2} \rp W_{\mathrm{o}}^{S_{\infty}}   
	\begin{pmatrix}
	\fra \fra ' \mathfrak{D}^{-2} & & \\
	& \fra ' \mathfrak{D}\- & \\
	& & 1 
	\end{pmatrix}.
	\end{align} 
	For our convenience, we denote $A_{\pi} (\fra) = A_{\pi} (\mathfrak{D}, \fra)$ when $r=3$. %For  $\gamma$,  $\gamma '\in \frO \smallsetminus \{0\}$,  the Fourier coefficients $A_{\pi} (\gamma)$, $A_{\pi} (\gamma', \gamma) $ will be defined in the same manner. We assume  that $ A_{\pi} ( 1 ) = 1$ or $ A_{\pi} ( 1, 1 ) = 1$ so that the Fourier coefficients are indeed Hecke eigenvalues.

	It is known from \cite[\S 17]{Qi-Bessel} that, when $\vv$ is archimedean, $\widetilde f_{\vv} (y) = \widetilde w_{\vv} \lp (-)^{r-1} y \rp \| y\|_{\vv}^{- \frac {r-1} 2}$ is the Hankel integral transform of $ f_{\varv} (x) = w_{\vv} \lp x \rp \| x\|_{\vv}^{ - \frac {r-1} 2}$ integrated against the Bessel kernel $J_{\pi_{\varv}} (xy)$ associated with $\pi_{\vv}$ (see \cite[(17.12), (17.13), (17.18), (17.19), (17.20)]{Qi-Bessel}), namely,
	\begin{equation}\label{2eq: Hankel transform tilde u and u}
	\widetilde f_{\vv} \lp  y \rp  =   \int_{F^\times_{\vv} }  J_{\pi_{\vv}} ( x y) f_{\vv} \lp x \rp   d x.  
	\end{equation}
	 We shall postpone the discussions on   $J_{\pi_{\varv}} (x)$ to \S \ref{sec: Bessel, zero} and \ref{sec: Bessel, infinity}.

	 For a nonzero ideal $\frb$ of $\frO$, we define  the additive character $\psi_{\frb}$ on $\frb\- / \frO  $  by $
	 \psi_{ \frb } (\gamma) = e \lp \Tr \,\gamma \rp $, 
	 with $\Tr : F / \frO \ra \BQ / \BZ$ the trace induced from that on $F$ (see  \cite[\S XIV.1]{Lang-ANT}). 
	 Moreover, let $F_{\frb}$ denote the ring of elements $\alpha$ in $F$ such that $ \|  \alpha \|_{\vv} \leqslant 1$ for all $\frp_{\vv} |\frb$, then $\psi_{\frb}$ extends to a character on $\frb\- F_{\frb}  / F_{\frb}  $ via the isomorphism  $F_{\frb} / \frb F_{\frb} \cong \frOO / \frb$.  Let $R = \{ \vv : \frp_{\vv} | \frb \}$. It follows from the definition of $\psi_{ \varv }$, $\varv \in S_f$, in \cite[\S XIV.1]{Lang-ANT} that $ \psi_{\frb} $ is the restriction of $\psi_R = \prod_{\varv \, \in R} \psi_{ \vv }$ on $\frb\- F_{\frb}  / F_{\frb} \cong \frb\- / \frO $ (embedded into $ \prod_{\varv \, \in R} \frb\- \frO_{\vv} / \frO_{\vv}$). 
	 
	 We now choose $\beta \in \Fx$ such that $\mathrm{ord}_{\vv} \, \beta = \mathrm{ord}_{\vv} \frb$ if $\frp_{\vv} | \frb$. For $\alpha \in F_\frb^{\times}$ we let    $\widebar \alpha $ denote a representative in $\frO \cap F_{\frb}^{\times}$ of the multiplicative inverse of $\alpha$ in $F_{\frb} / \frb F_{\frb} \cong \frOO / \frb$ so that $ \psi_{\frb} (\gamma/ \alpha \beta) = \psi_{\frb} (\widebar \alpha \gamma /  \beta)$ for all $\gamma \in F_{\frb}$.   For $\gamma, \gamma' \in F_{\frb}$, we define the Kloosterman sum
	 \begin{align}\label{2eq: Kloosterman sum}
	 S_{\frb} (\gamma, \gamma' ; \beta) = \sum_{ \sstyle \nu \, \in   \frOO /  \frb \atop \sstyle \| \nu \|_{\vv} = 1  \text{ if } \frp_{\vv} |\frb } 
	 \psi_{\frb} \lp   \frac {\gamma \nu +  \gamma '\, \widebar \nu} {\beta }     \rp.
	 \end{align}
	 We have   Weil's bound for the Kloosterman sum
	 \begin{align}\label{2eq: Weil's bound}
	 S_{\frb} (\gamma, \gamma' ; \beta) = O_{\varepsilon} \big( ( \RN \frb)^{  1 / 2 + \varepsilon} \big),
	 \end{align}
	 provided that $ \gamma $ and $\gamma' $ are relatively prime in $\frO_{\vv}$ if $ \frp_{\vv} |\frb $.
	  
\vskip 5 pt 

	 %We now choose $S = S_{\infty}$ and $\zeta \in \BA^{S_{\infty}}$ to be the diagonal embedding of a number in $ F$. 
	 %By some straightforward calculations,  
	 We may reformulate the \Voronoi summation formula in Proposition \ref{prop: Voronoi, adelic} as follows.
	 
	 \begin{prop}\label{2prop: classical Voronoi}
	 	Let notations be as above. Let $\alpha \in F$ and $\beta \in \Fx$ be such that   $\| \alpha\|_{\vv} = 1$ whenever $\|\beta \|_{\vv} < \|\alpha \|_{\vv}$, with $\vv \in S_f$. 
	 	Define $\frb =   \prod_{ \|\beta \|_{\vv} < \|\alpha \|_{\vv} } \frp_{\vv}^{\mathrm{ord}_{\vv} \, \beta}  $. 
	 	For $\vv \in S_{\infty}$, let $ f_{\vv} \in C_c^{\infty} (F_{\vv})$ and $\widetilde f_{\vv}$ be the Hankel transform of $ f_{\vv}$.
	 	Put $
	 	f = \prod_{\vv}   f_{\varv}  $ and $ \widetilde f = \prod_{\vv}  \widetilde f_{\varv}  $.  
	 	 Then, when $r=2$, we have 
	 	\begin{align}\label{2eq: Classical Voronoi, r=2}
	 	\sum_{ \gamma \, \in \frOO \hskip 0.5 pt ' \smallsetminus \{0\} } \hskip - 4 pt A_{\pi} ( \gamma \mathfrak{D} )   \psi_{f} \lp  \frac {\alpha \gamma} {\beta} \rp     f  (\gamma) =
	 	\frac 1 {  \RN \frb } \sum_{ \sstyle \gamma \, \in \Fx \atop \sstyle \gamma \frb^2 \subset \frO} \hskip - 4 pt
		 A_{\widetilde \pi} \lp \gamma \frb^2 \mathfrak{D} \rp \psi_{\frb} \lp - \frac {\widebar \alpha ( \beta^2 \gamma  ) } {\beta} \rp  \widetilde  f \lp   {\gamma}   \rp,
	 	\end{align} 
	 	and, when $r = 3$, we have 
	 	\begin{equation}\label{2eq: Classical Voronoi, r=3}
	 	\begin{split}
	 	 &\sum_{ \gamma \, \in \frOO \hskip 0.5 pt ' \smallsetminus \{0\} }    A_{\pi}  (\mathfrak{D}, \gamma \mathfrak{D})  \psi_{f} \lp   \frac {\alpha \gamma} {\beta} \rp f  (\gamma)  \\
	 	 & \hskip 6 pt =   \ \sum_{  \frb \, \subset \fra' \, \subset \frOO \hskip 0.5 pt '}     \frac { \RN \fra' } { (\RN \frb)^2}    
	 	\sum_{ \sstyle \gamma \, \in \Fx \atop \sstyle \gamma \frb^3 \subset \fra'^2 \frOO \hskip 0.5 pt ' }  
	 	A_{\widetilde \pi} \lp \fra' \mathfrak{D} , \gamma \frb^3 \fra'^{-2} \mathfrak{D} \rp S_{\frb \fra'^{-1}} \lp 1,  \widebar \alpha   \beta^3 \gamma/ \gamma'^2   ; \beta / \gamma' \rp 
	 	\widetilde  f \lp    {\gamma  }   \rp,
	 	\end{split}
	 	\end{equation} 
	 	in which $\gamma'  $ is an element in $F $ such that $\mathrm{ord}_{\vv} \gamma'   = \mathrm{ord}_{\vv}  \fra'$ whenever $\frp_{\vv } | \frb$.
	 \end{prop}
 
 \begin{proof}[Sketch of Proof] In the above settings, it is clear that  the left hand side of   \eqref{2eq: Voronoi, adelic} translates into that of \eqref{2eq: Classical Voronoi, r=2} or \eqref{2eq: Classical Voronoi, r=3}  for $r=2$ or $r=3$. As for the right hand side of \eqref{2eq: Voronoi, adelic}, we  just need to compute   the non-archimedean part, as the archimedean part has already be given by Hankel transforms as above. In order to be consistent with our normalization of Hankel transforms, we change $\gamma$ to $(-)^{r-1} \gamma$. In the following, we shall only consider the  case $r = 2$ and leave  to   readers the  similar but slightly more complicated computations for the case $r = 3$.%Note that  by definition $R = \left\{ \varv : \frp_{\varv} | \frb \right\}$.
 	
Invoking \eqref{2eq: K, r=2}, we infer that %for $\gamma \in F^{\times} \cap \frb^{-2} $
\begin{align*}
K_{R} (- \gamma, \alpha/\beta, \widetilde W_{\mathrm{o} \, R})  \widetilde W_{\mathrm{o}}^{S_{\infty} \cup R}  = \widetilde W_{\mathrm{o} \, R} \begin{pmatrix}
- \gamma \beta/\alpha & \\
& \alpha/\beta
\end{pmatrix}  \widetilde W_{\mathrm{o}}^{S_{\infty} \cup R}   
\begin{pmatrix}
- \gamma & \\
& 1 
\end{pmatrix} \cdot \psi_{R} \lp - \frac {\beta \gamma} {\alpha} \rp. 
\end{align*}
First, since $  \pi_{\varv} $ and $\widetilde \pi_{\varv} $ are adjusted to have trivial central character for $\varv \in S_f$ and $ \| \alpha\|_{\vv} = 1$  for all $\varv \in R$, the product of Whittaker functions is equal to 
\begin{align*}
\widetilde W_{\mathrm{o} \, R} \begin{pmatrix}
- \gamma \beta^2/\alpha^2 & \\
& 1
\end{pmatrix}  \widetilde W_{\mathrm{o}}^{S_{\infty} \cup R}   
\begin{pmatrix}
- \gamma & \\
& 1 
\end{pmatrix} & = \widetilde W_{\mathrm{o} \, R} \begin{pmatrix}
- \gamma \beta^2  & \\
& 1
\end{pmatrix}  \widetilde W_{\mathrm{o}}^{S_{\infty} \cup R}   
\begin{pmatrix}
- \gamma & \\
& 1 
\end{pmatrix} \\
& = \widetilde W_{\mathrm{o}  }^{S_{\infty}} \begin{pmatrix}
\gamma \frb^2  & \\
& 1
\end{pmatrix}   = \frac {A_{\widetilde \pi} \lp \gamma \frb^2 \mathfrak{D} \rp } {|\RN \gamma|^{1/2}  \RN \frb}   
\end{align*}
if $ \gamma \in F^{\times} \cap \frb^{-2} $ and vanishes otherwise. Second,  by the discussions on  $\psi_{ \frb }$  as above,  we have $ \psi_{ R } \lp -  {\beta  \gamma}/ {\alpha  } \rp = \psi_{ \frb } \lp -  {\beta^2 \gamma}/ {\alpha \beta} \rp = \psi_{ \frb } \lp  - \widebar \alpha {\lp \beta^2 \gamma \rp}/ { \beta} \rp$ for $ \gamma \in F^{\times} \cap \frb^{-2} $. Now we arrive at the right hand side of \eqref{2eq: Classical Voronoi, r=2}.
 \end{proof}

\begin{rem}\label{rem: Voronoi}
	When $F$ is of class number one, then we may choose $\alpha$ and $\beta$ to be a pair of integers that are relatively prime  and let $\frb = (\beta)$, $\fra' = (\gamma')$. In this way, upon changing $\gamma$ to $\gamma / \beta^2$ or $\gamma \gamma'^2 / \beta^3$, we obtain the classical \Voronoi summation formula as  in \cite{Miller-Schmid-2004-1, Miller-Schmid-2006}. Note that our normalization of Hankel transforms in \cite{Qi-Bessel} is slightly different  from   that in \cite{Miller-Schmid-2004-1, Miller-Schmid-2006}   in order to get the {\rm(}inverse{\rm)} Fourier transform when $r = 1$ and the classical Hankel transform when $r=2$. When the place $\varv$ is real,  $\|y\|_{\varv} \widetilde f_{\varv} \lp (-)^{r} y \rp$ is equal to the $F (y)$ in \cite[Theorem 1.18]{Miller-Schmid-2006} or \cite[Theorem 1.10]{Miller-Schmid-2009} {\rm(}if $f_{\varv} (x)$ is their $f (x)${\rm)}.
\end{rem}

\begin{rem}\label{rem: Voronoi, 2}
	In higher rank $r \geqslant 4$, following the computations in the proof above for $r = 2$, Ichino and Templier's  \Voronoi summation formula for $\GL_r $ may also be reformulated in the classical language, involving hyper-Kloosterman sum on the right. As can be seen in \cite{Miller-Schmid-2009}, the notations will be more complicated.
\end{rem}

	\section{Average of Fourier Coefficients} \label{sec: average}

	The Rankin-Selberg $L$-function
	\begin{equation*}
	L (s, \pi \otimes \widetilde{\pi}) = \left\{ \begin{split}
	& \sum_{ \fra \neq 0}  |A_{\pi} (\fra) |^2 (\RN \fra)^{- s}, \hskip 45 pt \text{ if } r = 2, \\
	& \mathop {\sum \sum}_{ \fra ,   \fra ' \neq 0} |A_{\pi} (\fra', \fra) |^2  \RN ( \fra \fra'^{\, 2}) ^{- s},  \hskip 10 pt \text{ if } r = 3, \end{split} \right. 
	\end{equation*}
	initially convergent for $s$ of real part   large, has a meromorphic continuation to the whole complex plane with only a simple pole at $s = 1$  (see \cite{J-S-Rankin-Selberg}). The Wiener-Ikehara theorem yields
	\begin{align*}
	& \sum_{ \RN \fra\, \leqslant X} |A_{\pi} (\fra) |^2 = O (X), \\
	 \mathop {\sum \sum}_{ \RN ( \fra \fra'^2)  \leqslant X } |A_{\pi} (\fra', \fra) |^2 & = O (X), \hskip 10 pt   \sum_{ \RN  \fra  \, \leqslant X } |A_{\pi} (\fra', \fra) |^2 = O \lp (\RN \fra')^2 X \rp .
	\end{align*}
	In addition, 
	It  follows from the Wiener-Ikehara theorem for the Dedekind zeta function of $F$ that
	$$\sum_{ \RN \fra \, \leqslant X} 1 = O (X). $$
	It should be remarked that Wiener-Ikehara would yield asymptotics and not just  upper bounds.
	Hence,   Cauchy-Schwarz implies
	\begin{align}
	\label{2eq: average over ideals, r = 2}  \sum_{ \RN \fra\, \leqslant X} |A_{\pi} (\fra) | & = O (X) , \\
	\label{2eq: average over ideals, r = 3}  \sum_{ \RN  \fra \, \leqslant X } |A_{\pi} (\fra', \fra) | & = O \lp  \RN \fra'   X \rp.
	\end{align}
	All the implied constants above depend only on $F$ and $\pi$. 
	More generally, for $0 \leqslant c < 1$, partial summation yields
	\begin{align}
	\label{2eq: general average over ideals, r = 2} \sum_{ \RN \fra \, \leqslant X} \frac {|A_{\pi} (\fra) |} {(\RN \fra)^{  c}}   =  & O_{c, \,\pi}  \lp X^{1 - c} \rp, \\
	\label{2eq: general average over ideals, r = 3}\sum_{ \RN  \fra \, \leqslant X } \frac {|A_{\pi} (\fra', \fra) |} { (\RN \fra)^{  c} } & = O_{c, \,\pi}   \lp  \RN \fra'   X^{1-c} \rp.
	\end{align}
	In the partial summation identity
	\begin{align*}
	\sum_{n=1}^X a_n b_n= \sum_{ n =1}^{X-1} S_n (b_{n} - b_{n+1}) + S_X b_X, \hskip 10 pt S_n = \sum_{ m = 1}^n a_m,
	\end{align*}
	if one chooses $a_n = \sum_{ \RN \fra = n } |A_{\pi} (\fra) |$ and $b_n = n^{- c}$, then \eqref{2eq: general average over ideals, r = 2} follows immediately from \eqref{2eq: average over ideals, r = 2}; one may prove \eqref{2eq: general average over ideals, r = 3}   in the same way.
	
	\vskip 5 pt
	
	Subsequently, we shall restrict ourselves to the integral ideals in a given ideal class, then both \eqref{2eq: general average over ideals, r = 2} and \eqref{2eq: general average over ideals, r = 3} hold if the sums are over these ideals.  More precisely, given an ideal $\fra$ of $\frO$, we have
	\begin{align}
	\label{2eq: general average over ideals, 2, r = 2} 
	\sum_{\sstyle [\gamma] \, \in \Fx / \frOO^{\times}  \atop
	\sstyle 1\leqslant |\RN \gamma| \RN \fra \, \leqslant  X } \frac {|A_{\pi} (\gamma \fra) |} { |\RN \gamma |^c }   =  & O_{c, \,\pi}  \lp (\RN \fra)^{  c} X^{1 - c} \rp, \\
	\label{2eq: general average over ideals, 2, r = 3}
	\sum_{\sstyle [\gamma] \, \in \Fx / \frOO^{\times}  \atop
		\sstyle 1\leqslant |\RN \gamma| \RN \fra \, \leqslant  X }  \frac {|A_{\pi} (\fra', \gamma \fra) |} {  |\RN \gamma|^{  c} } & = O_{c, \,\pi}   \lp  (\RN \fra' )  (\RN \fra)^c X^{1-c} \rp,
	\end{align}
	%We observe that there is a natural identification between  the  set of principal ideals  of $\frO$  and the set of orbits in $\frO \smallsetminus \{0\}$ under the  action of $\frOO^{\times}$ via multiplication. For $\gamma \in \frO \smallsetminus \{0\}$, 
	where   $[\gamma]$ is the orbit in $\Fx$ that contains $\gamma$, under the  action of $\frOO^{\times}$, namely, $ [\gamma] = \left\{ \gamma \epsilon : \epsilon \in \frOO^{\times} \right\} $.

	%Note that if $\fra$ is the ideal corresponding to $[\gamma]$, then $\RN \fra = |\RN \gamma |$. 
\delete{	Hence, we have
	\begin{align}
	\label{2eq: general average over principal ideals, r = 2} \sum_{ \sstyle [\gamma] \in \frO \smallsetminus\{0\} / \frOO^{\times} \atop \sstyle |\RN \gamma| \leqslant X } \frac {|A_{\pi} (\gamma) |} {|\RN \gamma|^{  c}}   =  & O_{c, \,\pi}  \lp X^{1 - c} \rp, \\
	\label{2eq: general average over principal ideals, r = 3}\sum_{\sstyle [\gamma] \in \frO \smallsetminus\{0\} / \frOO^{\times} \atop \sstyle |\RN \gamma| \leqslant X  } \frac {|A_{\pi} (\gamma', \gamma) |} { |\RN \gamma|^{  c} } & = O_{c, \,\pi}   \lp  |\RN \gamma' |  X^{1-c} \rp.
	\end{align}
}
	
	\begin{lem}\label{lem: Unit}
		Let $V > 0$. For $T \in \BR_+^{r_1 + r_2}$ define $ \RN (T) = \prod_{\vv} T_{\vv}^{N_{\vv}} $ and $ F_{\infty} (T) = \left\{ x \in F_{ {\infty}} : \| x \|_{\vv} \leqslant   T_{\vv}^{N_{\vv}} \right\}$. Suppose that $ V \RN (T)    \geqslant 1$. Then, for all $\gamma \in \Fx$ with $|\RN \gamma | \geqslant 1/ V$, we have
		\begin{align*}
		\mathrm{Card} \left\{ [\gamma]  \cap  F_\infty (T) \right\}  \lll 1 +   \log (V \RN (T))  ^{r_1 + r_2 - 1},
		\end{align*}
		with the implied constant depending only on $F$.
	\end{lem}
	
	\begin{proof}
		By Drichlet's unit theorem (see \cite[\S V.1]{Lang-ANT}), the rank of $\frOO^{\times}$ is $ r_1 + r_2 - 1$ and the torsion in $\frOO^{\times}$ is the finite cyclic group consisting of all the roots of unity in $F$. Define the map $ \log: F_{ {\infty}} \ra \BR^{r_1 + r_2} $ by 
		$\log x = ( \log \| x \|_{\vv} )_{\vv \, \in S_{\infty}} $. Restricting on  $\frOO^{\times}$, the image $ \log (\frOO^{\times}) $ is an $(r_1+r_2-1)$-dimensional lattice $\varLambda$ in the hyperplane $\BL^{r_1 + r_2 - 1} = \left\{ y \in \BR^{r_1 + r_2} : \sum y_{\vv} = 0 \right\}$ and $\log $ has a finite kernel, namely, the  torsion of $\frOO^{\times}$.
		
		We now want to count the number of points in the image of $[\gamma]  \cap  F_\infty (T)$ under the map $\log$. We have $\log [\gamma] = ( \log \| \gamma \|_{\vv} ) + \varLambda $ and it is a shifted lattice lying in the hyperplane $ \BL^{r_1 + r_2 - 1}_{\log |\RN \gamma |}$, with the definition $ \BL^{r_1 + r_2 - 1}_Y =  \{ y \in \BR^{r_1 + r_2} : \sum y_{\vv} = Y  \} $. On the other hand,  $\log F_\infty (T) =  \{ y \in \BR^{r_1 + r_2} : y_{\vv} \leqslant N_{\vv} \log T_{\vv}  \}$. Note that $ \BL^{r_1 + r_2 - 1}_{\log |\RN \gamma |} \cap \log F_\infty (T)$ are $(r_1 + r_2 - 1)$-dimensional simplexes of the same shape for all $\gamma \in \Fx$. Their volumes  do not exceed the volume of  $ \BL_{- \log V}^{r_1 + r_2 - 1}  \cap \log F_\infty (T)$ as $|\RN \gamma | \geqslant 1/ V$ and hence are uniformly $O \lp (\log V + \log \RN(T))^{r_1 + r_2 - 1} \rp$. Therefore, the number of the lattice points in $\log [\gamma]$ lying in the simplex $ \BL^{r_1 + r_2 - 1}_{ \log |\RN \gamma | } \cap \log F_\infty (T)$ is $O \big( 1 + \log (V \RN (T))^{r_1 + r_2 - 1} \big)$. %when $T$ is sufficiently large in terms of the covolume of $\varLambda$.
		 Finally, taking account of the finite torsion of $\frOO^{\times}$, which disappears under $\log$, it is clear that the number of points in $[\gamma]  \cap  F_\infty (T)$ is still $O \big(  1 + \log (V \RN (T))^{r_1 + r_2 - 1} \big)$. 
	\end{proof}

	\begin{lem}\label{lem: average over integers}
		Let notations be as above. %Suppose $\RN (T)$ is sufficiently large. 
		Let $0 \leqslant c < 1$. Suppose that $\RN \fra \RN(T) \geqslant 1$. Then
		\begin{align*}
		\sum_{\sstyle \gamma \, \in \, \Fx   \cap F_\infty (T) \atop \sstyle    \gamma \fra \, \subset \frO } 
		\frac { |A_{\pi}(\gamma \fra)|} {|\RN \gamma|^{c} } & = O \lp  (\RN \fra)^{1+\varepsilon} \RN(T)^{1-c+\varepsilon} \rp, \\
		\sum_{\sstyle \gamma \, \in \, \Fx   \cap F_\infty (T) \atop \sstyle  \gamma \fra \, \subset \frO }   
		\frac { |A_{\pi} (\fra', \gamma \fra) | } {|\RN \gamma|^{c}} & = O \lp          \RN \fra' (\RN \fra)^{1+\varepsilon} \RN(T)^{1-c+\varepsilon} \rp,
		\end{align*}
		with the implied constants depending only on $\varepsilon$, $c$, $F$ and $\pi$.
	\end{lem}
	\begin{proof} We shall only prove the first estimate as the second follows in the same way. By Lemma \ref{lem: Unit} and \eqref{2eq: general average over ideals, 2, r = 2},
		\begin{align*}
		\sum_{\sstyle \gamma \, \in \, \Fx   \cap F_\infty (T) \atop \sstyle    \gamma \fra \, \subset \frO } 
		 \frac { |A_{\pi}(\gamma \fra)| } {|\RN \gamma|^{c}} & = \sum_{\sstyle [\gamma] \in \Fx / \frOO^{\times} \atop \sstyle 1/ \RN \fra \, \leqslant  |\RN \gamma|  \, \leqslant   \RN(T) } 
		\frac { |A_{\pi}(\gamma \fra)|} {|\RN \gamma|^{c}} \cdot \mathrm{Card} \left\{ [\gamma]  \cap  F_\infty (T) \right\} \\
		& \lll \lp 1 +  \log (\RN \fra \RN(T))^{r_1 + r_2 - 1} \rp \sum_{\sstyle [\gamma] \in \Fx / \frOO^{\times} \atop \sstyle 1\leqslant |\RN \gamma| \RN \fra \, \leqslant \RN \fra \RN(T) } 
		\frac { |A_{\pi}(\gamma \fra)|} {|\RN \gamma|^{c}} \\ 
		& \lll \lp 1 +  \log (\RN \fra \RN(T))^{r_1 + r_2 - 1} \rp \RN \fra \RN(T)^{1-c}\\
		& \lll  (\RN \fra)^{1+\varepsilon} \RN(T)^{1-c+\varepsilon}.
		\end{align*}
	\end{proof}
	
		\begin{lem}\label{lem: average over integers, 2}
		Let notations be as above.   %Suppose $\RN (T)$ is sufficiently large. 
		For a nonempty subset $S \subset S_{\infty}$ and $T \in \BR_+^{r_1 + r_2}$ define $\|T\|_{S} = \prod_{ \vv \, \in S } T_{\vv}^{N_{\vv}}$ and $F_{\infty}^{S} (T) = \big\{ x \in F_{\infty} : \|x\|_{\vv} > T_{\vv}^{N_\vv} \text{ for all } \vv \in S, \|x\|_{\vv} \leqslant T_{\vv}^{N_\vv} \text{ for all } \vv \in S_{\infty} \smallsetminus S \big\}$.
		Let $0 \leqslant c < 1 < d$.  Suppose that $\RN \fra \RN(T) \geqslant 1$. Then, for any $0 < \varepsilon < d - 1$,
		\begin{align*}
		\sum_{\sstyle \gamma \, \in \, \Fx   \cap F_{\infty}^{S} (T) \atop \sstyle    \gamma \fra \, \subset \frO } 
		\frac { |A_{\pi}(\gamma \fra)|} {|\RN \gamma|^{c} \|\gamma \|_S^{d-c} } & = O \lp \frac { (\RN \fra )^{1 + \varepsilon} \RN(T)^{1-c + \varepsilon} }  {\|T\|_S^{d-c} } \rp, \\
		\sum_{\sstyle \gamma \, \in \, \Fx   \cap F_{\infty}^{S} (T) \atop \sstyle  \gamma \fra \, \subset \frO }   
		\frac { |A_{\pi} (\fra', \gamma \fra) | } {|\RN \gamma|^{c} \|\gamma \|_S^{d-c}  } & = O \lp \frac { \RN \fra ' (\RN \fra )^{1 + \varepsilon} \RN(T)^{1-c + \varepsilon} }  {\|T\|_S^{d-c} } \rp,
		\end{align*}
		with the implied constants depending only on $\varepsilon$, $c$, $d$, $F$  and $\pi$.
	\end{lem}

\begin{proof}
	%This lemma is an easy consequence of Lemma \ref{lem: average over integers, 2}.
	For $j   \in \BN^{|S|}$ we define $T_j \in \BR_+^{r_1 + r_2} $ by $T_{j \, \vv} = 2^{j_\vv} T_\vv$ if $\vv \in S$ and $T_{j \, \vv} =  T_\vv $ if $ \vv \in S_{\infty} \smallsetminus S $. We introduce a dyadic partition of $F_{\infty}^{S} (T)$,
	\begin{align*}
	F_{\infty}^{S} (T) = \bigcup_{j \, \in \BN^{|S|} } F_{\infty}^{j} (T), 
	\end{align*}
	with $
	F_{\infty}^j (T) = \left\{ x \in F_\infty :   T_{j \, \vv}^{N_\vv} < \|x\|_{\vv} \leqslant 2^{N_{\vv}} T_{j \, \vv}^{N_\vv} \text{ if } \vv \in S, \|x\|_{\vv} \leqslant T_{\vv}^{N_\vv} \text{ if } \vv \in S_{\infty} \smallsetminus S 
	 \right\}.$  Then
	\begin{align*}
	\sum_{\sstyle \gamma \, \in \, \Fx   \cap F_{\infty}^{S} (T) \atop \sstyle    \gamma \fra \, \subset \frO } 
	\frac { |A_{\pi}(\gamma \fra)|} {|\RN \gamma|^{c} \|\gamma \|_S^{d-c} } & = \sum_{ j \,\in \BN^{|S|} } \sum_{\sstyle \gamma \, \in \, \Fx   \cap F_{\infty}^j (T) \atop \sstyle    \gamma \fra \, \subset \frO } 
	\frac { |A_{\pi}(\gamma \fra)|} {|\RN \gamma|^{c} \|\gamma \|_S^{d-c} } \\
	& \leqslant \sum_{ j \,\in \BN^{|S|} } 
	\frac 1 { \ \left \| T_j \right \|_S^{d-c} } 
	\sum_{\sstyle \gamma \, \in \, \Fx   \cap F_{\infty} (T_{j+1} ) \atop \sstyle    \gamma \fra \, \subset \frO } 
	\frac { |A_{\pi}(\gamma \fra)|} {|\RN \gamma|^{c}  } \\
	& \lll (\RN \fra)^{1 + \varepsilon} \sum_{ j \,\in \BN^{|S|} } 
	\frac {\RN(T_{j+1})^{1-c + \varepsilon}} { \ \left \| T_j \right \|_S^{d-c} }   \\
	& = \frac { (\RN \fra )^{1 + \varepsilon} \RN(T)^{1-c + \varepsilon} }  {\|T\|_S^{d-c} } \sum_{ j \,\in \BN^{|S|} } \frac {2^{(1-c+\varepsilon) \sum_{\vv \,\in S } N_{\vv }} } { 2^{(d-1-\varepsilon) \sum_{\vv \,\in S } N_{\vv } j_\vv } } \\
	& \lll \frac { (\RN \fra )^{1 + \varepsilon} \RN(T)^{1-c +\varepsilon} }  {\|T\|_S^{d-c} } ,
	\end{align*}  
	where Lemma \ref{lem: average over integers, 2} is applied for the second inequality. This proves the first estimate and the second follows in the same way.
\end{proof}

\section{Asymptotic of Bessel Kernels and Estimates for Hankel Transforms}\label{sec: asymptotic}

In this section, we shall exploit the analytic results on the asymptotic behaviours of Bessel kernels in \cite{Qi-Bessel}  to study Hankel transforms, which arise as the analytic component of the \Voronoi summation formula.

Throughout this section, we shall suppress $\vv$ from our notations. Accordingly, $F$ will be an archimedean local field, and $\pi$ will be a unitary irreducible admissible  generic representation of $\GL_r (F)$. %arising from an irreducible cuspidal automorphic representation of $\GL_r$. 
With $\pi$ are associated certain embedding parameters $(\umu, \udelta) \in \BC^{r } \times (\BZ/2\BZ)^r$ if $F = \BR$ or  $(\umu, \um) \in \BC^{r} \times \BZ^r$ if $F = \BC$. Unitaricity implies that $\umu \in \BL^{r-1} = \left\{\ulambda \in \BC^r : \sum_{l      = 1}^r \lambda_l      = 0 \right \}$. With the notations of \cite{Qi-Bessel}, we shall denote the Bessel kernel   by $J_{(\umu, \udelta)} (x)$ or $J_{(\umu, \um)} (z)$. Here and henceforth, $x$, $y$ will always stand for real variables,  while $z$, $u$ for complex variables.

When $r = 2$, our Bessel kernels for $\GL_2 (F)$ may be expressed in terms of classical Bessel functions. For illustration purposes, let us  take the   simplest examples for $\GL_2 (\BR)$ and $\GL_2 (\BC)$. For the discrete series representation $\sigma (m)$ of $\GL_2 (\BR)$, its associated Bessel kernel is equal to $2 \pi i^{m+1} J_{m} \lp 4 \pi \sqrt x \rp$ on $  \BR_+$ and vanishes identically on $- \BR_+$. %The formula of Bessel functions for $\GL_2 (\BC)$ are more complicated. 
For the {spherical} principal series representation $\pi^+_0 (\mu)$ of $\GL_2 (\BC)$, % induced from the character $\chiup_{\mu} \otimes \chiup_{\mu}\-$, with $\chiup_{\mu} = |z|^{2 \mu}$, 
the   Bessel function reads
\begin{align*}
\frac {2 \pi^2} {\sin (2\pi \mu)}  \lp J_{- 2\mu   } \lp 4 \pi \sqrt z \rp J_{- 2\mu    } \big( 4 \pi \sqrt {\overline z} \big)    - J_{ 2\mu   } \lp 4 \pi \sqrt z \rp J_{ 2\mu }  \big( 4 \pi \sqrt {\overline z} \big)     \rp.
\end{align*}
This expression is well-defined and   should be in its limit form when  $ \mu $ is a half integer.
Here   $J_{\nu} (z)$ is the Bessel function of the first kind of order $\nu$ as usual. For a complete list of the formulae of  the Bessel kernels for all infinite dimensional irreducible representations of $\GL_2 (\BR)$ and $\GL_2 (\BC)$, see \cite[\S 4.3, 15.3]{Qi-Bessel} and also \cite[\S 18]{Qi-Bessel}.

\subsection{Bounds for Bessel Kernels near Zero}\label{sec: Bessel, zero}
%\ 
%\vskip 5 pt

%\subsubsection{Embedding Parameters}\label{sec: embedding parameters}
The asymptotic behaviour of the Bessel kernel $J_{\pi}$ near zero is very close to its associated parameters, so it is worthwhile to first describe them in more details.

When $F=\BR$, there are essentially two types of representations of $\GL_r (\BR)$ (see the discussions in \cite[\S 2]{Miller-Wilton}).   First, $\pi$ may be a fully induced principal series representation, in which case $|\Re \mu_l| < \frac 1 2$ for all $l = 1, ..., r$. Second,
when $r=2$,  $\pi$ is the discrete series $\sigma (m)$ of $\GL_2 (\BR)$, $m = 1, 2, ...$, or, when $r=3$, $\pi$ is an induced representation of $\GL_3(\BR)$ constructed from a discrete series  $\sigma (m)$ of $\GL_2 (\BR)$; in the former case one may choose $$\mu_1 =  \frac m 2, \ \mu_2 = - \frac m 2, \hskip 10 pt \delta_1+ \delta_2 \equiv m+1 (\mod 2),$$
and in the latter, $$\mu_1 =   it + \frac m 2, \ \mu_2 = -2it, \ \mu_3 = it - \frac m 2, \hskip 10 pt \delta_1+ \delta_3 \equiv \delta_2 \equiv m+1 (\mod 2).$$

When $F=\BC$, it is known that $\pi$ must be a fully induced principal series representation and $|\Re \mu_l| < \frac 1 2$ for all $l = 1, ..., r$.

\begin{lem}\label{lem: bound, at zero}
	Let  
	$C >0$. When $F = \BR$, for   $x \in \BR^{\times}$ with $ |x| < C$, we have
	\begin{align*}
	  J_{(\umu, \udelta)} \lp  x \rp  \lll |x|^{-   1 / 2 }.
	\end{align*}
	Similarly, when   $F = \BC$, for  $z \in \BC^{\times}$ with $ |z| < C$, we have
	\begin{align*}
	  J_{(\umu, \um)} \lp  z \rp  \lll |z|^{- 1  }.
	\end{align*}
	The implied constants above depend  only on $ C$  and the parameters $(\umu, \udelta)$ or $(\umu, \um)$. %{\rm (}the embedding parameters of $\pi${\rm )}.
\end{lem}
We shall prove this lemma by bounding the Mellin-Barnes type integrals of certain gamma factors as in \cite[\S 1.1, 3]{Qi-Bessel}. 

\begin{rem}
	When  the parameters are generic in the sense that all the poles of the gamma factors are simple,  sharper bounds  may be easily seen from the power series expansions of the Bessel functions, which are obtained from shifting the integral contours in  the Mellin-Barnes type integrals to the far left and collecting the residues. In deed, one has  
	\begin{align*}
J_{(\umu, \udelta)} \lp  x \rp \lll \max \left\{ \left| |x|^{\,\mu_{l} } \right| \right\}, \hskip 10 pt J_{(\umu, \um)} \lp  z \rp \lll \max \left\{ \left| |z|^{2\mu_{l} } \right| \right\}.
	\end{align*} In the non-generic cases, with some work, it may be shown that only powers of $\log |x|$ or $\log |z|$ will be lost. For more details, see \cite[\S 11.1, 11.2, 15.1]{Qi-Bessel}.
\end{rem}

%\vskip 5 pt

\subsubsection{Bounds for Real Bessel Kernels near Zero}\label{sec: bound at zero, real}
Suppose $F=\BR$. According to \cite[\S 1.1, 3]{Qi-Bessel}, let
\begin{equation*}% \label{1def: G delta}
G_\delta (s) = i^\delta \pi^{ \frac 1 2 - s} \frac {\Gamma \lp \frac 1 2 ({s + \delta} ) \rp} {\Gamma \lp \frac 1 2 ({1 - s + \delta} ) \rp} = 
\left\{ \begin{split}
& 2(2 \pi)^{-s} \Gamma (s) \cos \left(\frac {\pi s} 2 \right), \hskip 10pt \text { if } \delta = 0,\\
& 2 i (2 \pi)^{-s} \Gamma (s) \sin  \left(\frac {\pi s} 2 \right), \hskip 9 pt \text { if } \delta = 1,
\end{split} \right.
\end{equation*}
%form the product
\begin{equation*}%\label{1def: G (lambda, delta)}
G_{(\umu, \udelta) } (s) = \prod_{l      = 1}^r G_{\delta_{l     } } (s - \mu_l     ),
\end{equation*}
and  define
\begin{align*}
j_{(\umu, \udelta)} (x) =  \frac 1  {2 \pi i} \int_{\EC_{(\umu, \udelta)}} G_{(\umu, \udelta)} (s ) x^{- s} d s, \hskip 10 pt x > 0,
\end{align*}
where $\EC_{(\umu, \udelta)}$ is a curve obtained from a vertical line on  the right of  all the poles of $G_{(\umu, \udelta)} (s )$ with two ends at infinity shifted to the left, say, to $\Re s = - 1$ (see \cite[Definition 3.2]{Qi-Bessel}); the shift is needed only to secure convergence.
The Bessel kernel $J_{\pi} = J_{(\umu, \udelta)}$ is defined as
\begin{equation*}%\label{3def: Bessel function, R, 0}
\begin{split}
J_{(\umu, \udelta)} \lp \pm x \rp = \frac 1 2 \lp j_{(\umu, \udelta)} (x) \pm j_{(\umu, \udelta + \ue)} (x) \rp,
\end{split}
\end{equation*}
where $\ue = (1, 1)$ if $r =2$ and $\ue = (1,1,1)$ if $r=3$.

We remark that both $G_{(\umu, \udelta) } (s)$ and $ G_{(\umu, \udelta + \ue) } (s) $ have no poles in the half plane $\Re s \geqslant \frac 1 2$. This is clear when $\pi$ is a principal series. When $\pi$ is the discrete series $\sigma (m)$ of $\GL_2 (\BR)$, by the duplication formula and Euler's reflection formula of the Gamma
function, both $G_{(\umu, \udelta) } (s)$ and $ G_{(\umu, \udelta + \ue) } (s)$ are equal to
\begin{align*}
i^{m + 1} (2\pi)^{1-2 s } \frac { \Gamma \lp s + \frac 1 2{m}   \rp} { \Gamma \lp 1 - s + \frac 1 2{m}   \rp },
\end{align*}
which has no poles even in the larger region $\Re s > - \frac 1 2 m$. We are in a similar situation, if $\pi$ is a representation of $\GL_3 (\BR)$ coming from  the discrete series $\sigma (m)$ of $\GL_2 (\BR)$. Therefore, we may and do choose the contours $ \EC_{(\umu, \udelta)} $ and $ \EC_{(\umu, \udelta+\ue)} $ to be contained in the  half plane $\Re s \leqslant \frac 1 2$. %with $A \geqslant   \frac 1 2$, 
Then, some simple estimations by Stirling's formula of the  integrals that define $j_{(\umu, \udelta)} (x)$ and $j_{(\umu, \udelta + \ue)} (x)$ yield the bound for $ J_{(\umu, \udelta)} \lp \pm x \rp $ in Lemma \ref{lem: bound, at zero}. %With similar arguments, one may prove Lemma \ref{lem: bound, at zero} for the derivatives of $ J_{(\umu, \udelta)}  $. 

\vskip 5 pt

\subsubsection{Bounds for Complex Bessel Kernels near Zero}\label{sec: zero complex}

Now let $F = \BC$. Define 
\begin{equation*}%\label{1def: G m (s)}
G_m (s) = i^{|m| } (2\pi)^{1-2 s } \frac { \Gamma \lp s + \frac 1 2{|m|}   \rp} { \Gamma \lp 1 - s + \frac 1 2{|m|}   \rp },
\end{equation*}
\begin{equation*}%\label{1def: G (mu, m)}
G _{(\umu, \um)} (s) = \prod_{l      = 1}^r G_{m_{l     } } (s - \mu_l     ),
\end{equation*}
and
\begin{equation*}%\label{3def: Bessel kernel j mu m}
j_{(\umu, \um)} (x) = \frac 1  {2 \pi i} \int_{\EC _{ (\umu, \um)}} G_{(\umu, \um)} (s ) x^{- 2 s} d s,  \hskip 10 pt x > 0,
\end{equation*}
where   $    \EC _{ (\umu, \um)}$ %, the curve $ \EC _{ (\umu, \um)}$ scaled by $2$, 
is defined similarly as $\EC _{ (\umu, \udelta)}$ such that all the poles of $G_{(\umu, \um)} (s )$ stay on the left (see \cite[Definition 3.2]{Qi-Bessel}). We define the Bessel kernel $J_{\pi} = J_{(\umu, \um)}$  in the polar coordinates,
\begin{equation*}%\label{2eq: Bessel kernel over C, polar}
J_{(\umu, \um)} \lp x e^{i \phi} \rp =  \frac 1 {2 \pi} \sum_{m = -\infty}^{\infty} j_{(\umu, \um + m\ue )} (x ) e^{ i m \phi}.
\end{equation*}

Now assume that $\max\{|\Re \mu_l | \} < \frac 1 2$ and set $M = \max \{|m_l|\}$.  Let $ x < C$. We truncate the Fourier series at $|m| = M$.  When $|m| > M$, the estimate in \cite[Lemma 3.10]{Qi-Bessel} by Stirling's formula yields
\begin{align*}
j_{(\umu, \um + m\ue )} (x ) \lll \lp \frac { (2 \pi e)^{|m|}  } {(|m| + 1)^{|m| - M}} \rp^{r }     x^{ |m| - M - 1}.
\end{align*}
 When $|m| \leqslant M$, we choose the contours $ \EC_{(\umu, \udelta + m \ue)} $  in the half plane $\Re s \leqslant \frac 1 2$, so  
\begin{align*}
j_{(\umu, \um + m\ue )} (x ) \lll  x^{ - 1} .
\end{align*}
Combining these estimates, we obtain
\begin{align*}
J_{(\umu, \um)} \lp x e^{i \phi} \rp \lll x^{-1}.
\end{align*}
All the implied constants above depend only on $C$ and $(\umu, \um)$.  \delete{ This proves Lemma \ref{lem: bound, at zero} for $j = j' = 0$.
In general, we may apply similar arguments to show that
\begin{align*}
(\partial/\partial x)^j (\partial/\partial \phi)^k J_{(\umu, \um)} \lp x e^{i \phi} \rp \lll x^{-1 - j},
\end{align*}
and then establish the bounds for $ J_{(\umu, \um)} $ in the Cartesian coordinates.}

\subsection{Asymptotic of Bessel Kernels at Infinity}\label{sec: Bessel, infinity}

We now collect some results  from \cite{Qi-Bessel} on the asymptotic of Bessel kernels at infinity. See \cite[Theorem 14.1, 16.6,  16.7]{Qi-Bessel}. %Instead of asymptotic expansions, the following formulations are more convenient for our purpose.

\begin{prop}\label{6prop: asymptotic} 
	Let $K \geqslant 0$.
	
	When $F = \BR$, 
	for $x > 0$, we may write  
	\begin{align*}
	& J_{(\umu, \udelta)} \left(x^2 \right)   =  \sum_{\pm}  \frac { { e \lp   \pm  2 x   \rp }} {  x^{1/2} }   W_{ (\umu, \udelta) }^{\pm} (x) + E^+_{(\umu, \udelta)} (x), \\
	& J_{(\umu, \udelta)} \left( - x^2 \right)   = E^-_{(\umu, \udelta)} (x),  
	\end{align*}
	if $r = 2$,
	\begin{align*}
	J_{(\umu, \udelta)} \left(\pm x^3 \right) & =     \frac { { e \lp   \pm  3 x   \rp }}  x  W_{ (\umu, \udelta) }^{\pm} (x) + E^\pm_{(\umu, \udelta)} (x),   
	\end{align*}
	if $r = 3$, and there is a constant $C_{K, \, \umu}$, depending only on $K$ and $\umu$, such that for $x \geqslant C_{K, \, \umu }$, we have
	\begin{align}\label{6eq: asymptotic of W, real}
	W_{ (\umu, \udelta) }^{\pm } (x) =    \sum_{k= 0}^{K-1}  B^{\pm}_{k }   x^{-   k  }  
	+ O_{K, \,\umu } \left(  x^{- K  }\right), 
	\end{align} 
	with the coefficients $B_k^{\pm}$ depending only on $ (\umu, \udelta)$, and
	\begin{equation*}%\label{6eq: E (x)}
	E^{\pm }_{(\umu, \udelta)} \left( x \right) =  O_{ \umu } \big( x^{-  ( {r-1} ) / {2 } } \exp \lp {  - 2 \pi r \sin \lp   \pi / r \rp  x } \rp \big).
	\end{equation*}

	Define the sectors
	\begin{align*}
	\BS_2  & =   
	\left\{ z \in \BCx : -  \tfrac 1 2 \pi \leqslant \arg z < \tfrac 1 2  \pi   \right\} , \\
	\BS_3  & = \left\{ z \in \BCx : - \tfrac 2 3 \pi  \leqslant \arg z < 0   \right\} .
	\end{align*} When $F = \BC$, we may write
	\begin{align*}
	J _{(\umu, \um)} \left(z^r \right) = \sum_{\xi^r = 1} \frac { e \big( r \big( \xi z +   \overline {\xi z } \big) \big) }   {  |z|^{r-1} [\xi z]^{|\um|} } W_{(\umu, \um)} \lp  z, \xi    \rp + E_{(\umu, \um)} (z),
	\end{align*} 
	with   notations $[z] = z/ |z|$ and $|\um| = \sum_{l=1}^r m_l$,
	 and there is a constant $C_{K, \,(\umu, \um)}$, depending only on $K$ and $(\umu, \um)$, such that, for  $z \in  \BS_r  $ with $|z| \geqslant C_{K, \,(\umu, \um)} $, we have 
	\begin{align}\label{6eq: asymptotic of W, complex}
	W_{(\umu, \um)} \lp z , \xi \rp =    
	\mathop{\sum\sum}_{\substack{ k,\, k' = 0, ..., K-1 \\ k + k' \leqslant K-1 }} B_{k, k'}   \xi^{- k + k'}   z^{- k } \overline z^{- k' }  + O_{K, (\umu, \um)} \lp |z|^{- K } \rp ,
	\end{align} 
	with the coefficients $B_{k, k'}$ depending only on $ (\umu, \um)$, 
	\begin{align*}
	E_{ (\umu, \um) } (z) = 0,
	\end{align*} 
	if $r = 2$, and
	\begin{align*}
	E_{ (\umu, \um) } (z) = O_{(\umu, \um)} \lp |z|^{- 2} \exp \lp  - 12 \pi   \sin \lp \tfrac 1 3 \pi \rp \sin   \lp \tfrac 1 {12} \pi \rp |z| \rp \rp,
	\end{align*} 
	if $r = 3$.
\end{prop}

While Stirling's asymptotic formula may be applied to  the Mellin-Barnes type integrals in \S \ref{sec: bound at zero, real} to prove the asymptotic formula for $J_{(\umu, \udelta)} (x)$ as in   Proposition \ref{6prop: asymptotic} (see \cite{XLi}, \cite{Blomer} and \cite[Appendix B]{Qi-Bessel}), this method does not work for $J_{(\umu, \um)} (z)$ because it is  an {\it infinite}   series  of Mellin-Barnes type integrals as one has seen in \S \ref{sec: zero complex}. In \cite{Qi-Bessel}, the author finds two other approaches to   the asymptotic formulae for the {\it analytic continuations} of  certain Bessel functions initially defined on $\BR_+$; they also yield the first part of  Proposition \ref{6prop: asymptotic}. The first is to apply stationary phase to the {\it formal} integral representation of   Bessel functions (this method is even new for classical Bessel functions). The second is to apply the asymptotic theory of differential equations to the Bessel equations. As a result, one obtains explicit connection  formulae between the two kinds of Bessel functions associated with the two singularities $0$ and $\infty$ of the Bessel equations. In order to study  $J_{(\umu, \um)} (z)$, one may establish  two formulae that express $J_{(\umu, \um)} (z)$ in terms of the two kinds of  Bessel functions respectively; their  connection  formulae play a vital role here in the establishment. Finally, the asymptotic formula for  $J_{(\umu, \um)} (z)$   as in Proposition \ref{6prop: asymptotic} follows from the second formula along with the asymptotic formula derived from the theory of differential equations. Readers are referred to \cite[Chapter 2, 3]{Qi-Bessel} for more details.
%of the second kind These connection formulae are subsequently applied  Finally, the asymptotic formula for  may be derived from a formula that relates it to the Bessel functions of the second kind.

\subsection{Estimates for Hankel Transforms}

%We now consider Hankel transforms of compactly supported   functions. 
Recall that Hankel transforms are defined as
\begin{align}\label{6eq: Hankel}
\widetilde f (y) = \int_{\BR} J_{(\umu, \udelta)} (x y) f (x) d x, \hskip 10 pt \widetilde f (u) = \sideset{}{_{\,\BC}}{\iint}  J_{(\umu, \um)} ( z u) f (z)  \, i d z \nwedge d \overline z.
\end{align}
For a     weight function $w  $ given as in Theorem \ref{thm: main, smooth} and a scalar $\rho \in F$, our choice of the test function will be
\begin{align} 
\label{6eq: choice of f, real} & f  (x) = w  (x) e \lp - \rho  x \rp, \hskip 28 pt \text{ if } F  = \BR,\\
\label{6eq: choice of f, complex} & f  (z) = w  (z) e \lp  - \rho  z - \overline \rho \overline z \rp, \hskip 10 pt \text{ if } F = \BC.
\end{align}

%We shall assume that $Y =    T |\rho| \ggg 1$. %so that there are oscillations in the exponential factor.

\begin{lem} \label{lem: bound for Hankel, 1}
	Let $C > 0$.  Let $y \in \BR^{\times}$ and $u \in \BCx$. Suppose that  $ 0 < T |y|, \, T |u| \leqslant C $.
	Then,  
	\begin{align*}
	\widetilde f (y) \lll_{ C,\, \umu,\, \varDelta}   T^{1/2}   |y| ^{- 1/2},
	\end{align*} 
	if  $F = \BR$, and
	\begin{align*}
	\widetilde f (u) \lll_{ C, (\umu, \um), \, \varDelta}   T    |u| ^{- 1 },
	\end{align*} 
	if  $F = \BC$. 
\end{lem}

\begin{proof}
	This lemma immediately follows from Lemma \ref{lem: bound, at zero}.
\end{proof}

Next, we consider $\widetilde f $ of large argument. For this we have two lemmas for the cases $T |\rho| \ggg 1$ and $ T|\rho | \lll 1$ separately.

\begin{lem} \label{lem: bound for Hankel, 2}
	Let $K \geqslant 0$. Put $C = C_{K,\, \umu}$ or $C_{2 K, (\umu, \um)}$ as in Proposition {\rm\ref{6prop: asymptotic} }. Let $y \in \BR^{\times}$ and $u \in \BCx$.  Suppose that $T |\rho| $ is sufficiently large, say $ T |\rho|  > \varDelta^{r-1}C^{1/r}$, and that $  T |y|, \, T |u| > C  $. 
	
	When $F=\BR$, we have 
	\begin{align*}
	\widetilde f (y) \lll_{ \umu, \, \varDelta}  T^{1/2}   |y|^{- 1/2},
	\end{align*} 
	if  $(T |\rho|)^{r } /\varDelta^{r (r-1)}  \leqslant T |y| \leqslant \varDelta^{r(r-1)} (T |\rho|)^{r }  $, and
	\begin{align*}
	\widetilde f (y) \lll_{K,\, \umu, \, \varDelta}    {T }   {\lp T |y| \rp^{-(r-1)/2r - K/ r} },
	\end{align*} 
	if otherwise.
	
	When $F=\BC$, we have 
	\begin{align*}
	\widetilde f (u) \lll_{ (\umu, \um), \, \varDelta}  T    |u|^{- 1 },
	\end{align*} 
	if  $(T |\rho|)^{r } /\varDelta^{r (r-1)}  \leqslant T |u| \leqslant \varDelta^{r(r-1)} (T |\rho|)^{r }  $, and
	\begin{align*}
	\widetilde f (u) \lll_{K,\, (\umu, \um),\, \varDelta}    {T^2 }   {\lp T |u| \rp^{-(r-1)/ r - 2 K/ r} },
	\end{align*} 
	if otherwise.
\end{lem}

\begin{proof}
	We first consider the real case. In view of Proposition \ref{6prop: asymptotic}, there are four similar integrals coming from the leading terms of the asymptotic expansions of $W_{ (\umu, \udelta) }^{ \pm } (x)$ as in \eqref{6eq: asymptotic of W, real}, two of them being of the following form, up to   constant multiple,
	\begin{align*}
	%B_0^+ \int_{ R}^{2 R} \frac {e \big( r (x y)^{ 1 / r} \big)} {(x y)^{(r-1)/ 2r}} f (x) d x = 
	\frac { 1 } {y^{(r-1)/ 2r} } \int_{T^{  1 / r}}^{(\varDelta T)^{  1 / r}}   {e \big(  \pm r y^{1/r} x - \rho x^r \big)}  x^{(r-1) / 2} w (x^r)  d x, 
	\end{align*}
	in which we assume that $\rho, y > 0$. Upon changing the variable $x$ to $  y^{1/ r (r-1)} x / \rho^{1/(r-1)}    $, we obtain the integral
	\begin{align}\label{6eq: integral, real}
	h_{\pm} (\lambdaup) = \frac {T^{(r-1) / 2 r}} {\rho^{(r+1) / 2r} \lambdaup^{(r^2-2r-1) / 2r} } \int_{ X }^{\varDelta^{1/r} X } e \lp \lambdaup p_{\pm} (x) \rp g (x, \lambdaup) d x ,
	\end{align}
	with  
	\begin{align*}
	& \lambdaup = \lp \frac y   \rho \rp^{1/{(r-1)}}, \hskip 10 pt X %= \frac {T^{1/r} \rho^{1/(r-1)}} {y^{1/r(r-1)}} 
	= \lp \frac {T \rho} {\lambdaup} \rp^{1/r}, \\
	\hskip 1 pt p_{\pm} (x) =   \pm \ & r x -   x^r   , \hskip 10 pt  g (x, \lambdaup) = \lp \frac {\lambdaup x^r} {T \rho} \rp^{(r-1)/ 2 r} w \lp \frac {\lambdaup x^r} {\rho} \rp,
	\end{align*}
	and the weight function $g (x, \lambdaup)$ satisfying 
	\begin{align*}
	(\partial / \partial x)^j (\partial / \partial \lambdaup)^k g (x, \lambdaup) \lll X^{-j}\lambdaup^{- k}.
	\end{align*}
	%Note that the derivative of the phase function is $ r (x^{r-1} \pm 1)  $.  
	For $h_- (\lambdaup)$, we have $ p_{-}' (x) = - r (x^{r-1} + 1) \leqslant - r$ for all $x \in [X, \varDelta^{1/r} X]$, and hence by repeating partial integration $K$ times, we have the estimate
	\begin{align*}%\label{6eq: bound for h+}
	h_- (\lambdaup) \lll \frac {T^{(r-1) / 2 r} X} {\rho^{(r+1) / 2r} \lambdaup^{(r^2-2r-1) / 2r} } \frac 1  {\lp X\, \lambdaup  \rp^{K } } =  \frac {T^{(r+1)/2r}} {\lp \rho \, \lambdaup^{r-1} \rp^{(r-1)/2r} } \frac 1  {\lp X\, \lambdaup  \rp^{K } } = \frac {T }   {\lp T y \rp^{(r-1)/2r + K/ r} }.
	\end{align*}
	As for $h_+ (\lambdaup)$, the stationary point of the phase function $p_+ (x)$ satisfies $ 1 - x^{r-1}   = 0 $, giving $x_0 = 1$. Since $ |p'_+ (x) | > r \big( 1 - \varDelta^{- (r-1)^2 / r} \big)$ on $[X, \varDelta^{1/r} X]$ if either $X < 1/ \varDelta$ or $X >  \varDelta$, repeating partial integration yields the same bound for $ h_+ (\lambdaup) $ as above. When $1/\varDelta \leqslant X \leqslant \varDelta $, so that $T^{ r-1 } \rho^{r } /\varDelta^{r (r-1)}  \leqslant y \leqslant \varDelta^{r(r-1)} T^{r-1} \rho^{r }   $,   the stationary phase method (see for example \cite[Theorem 1.1.4]{Sogge}  or \cite[Theorem 7.7.5]{Hormander}) may be applied to bound the integral in \eqref{6eq: integral, real} by $\lambdaup^{-1/2}$ (note that $\lambdaup \ggg T \rho \ggg 1$), then
	\begin{align*}
	h_+ (\lambdaup) \lll \frac {T^{(r-1) / 2 r}} {\rho^{(r+1) / 2r} \lambdaup^{(r^2-2r-1) / 2r} } \, \frac 1 {\lambdaup^{1/2}} =  \frac {T^{(r-1) / 2 r}} {\rho^{1 / 2 (r-1)} y^{(r^2 - r -1) / 2r (r-1)} } \lll \frac {T^{1/2}} {y^{1/2}}.
	\end{align*}
	Moreover, the contributions from lower-order terms may be treated in the same manner and those from the error term in \eqref{6eq: asymptotic of W, real} and $E^{\pm }_{(\umu, \udelta)} \left( x \right) $ can also be bounded by $  {T }   {\lp T y \rp^{-(r-1)/2r - K/ r} }$.

	The complex case is similar. Thus we shall only consider the total leading-term contributions from the asymptotic expansions of all $W_{ (\umu, \um) } (z, \xi)$ as in \eqref{6eq: asymptotic of W, complex}.
	Without loss of generality, we assume $\rho > 0$. The integral that we need to estimate is
	\begin{align*}
	& \frac { 1} {|u|^{(r-1)/ r}  [ u^{1/r}  ] ^{|\um| } } \sum_{ \xi^r = 1} \sideset{}{_{\, \BS_r}}{\iint}   { e \big( r \big( \xi u^{1/r} z +   \overline {\xi u^{1/r} z } \big)   - \rho \lp z^r + \overline z^r \rp \big) } \frac  {|z|^{r-1}} {\, [\xi z]^{  |\um|}} w (z^r) \, i dz \nwedge d \overline z    \\
	= \ &   \frac {1} {|u|^{(r-1)/ r} [ u^{1/r}  ] ^{|\um| } }   \sideset{}{_{\, \BC}}{\iint}   { e \big( r \big(   u^{1/r} z +   \overline {  u^{1/r} z } \big) - \rho \lp z^r + \overline z^r \rp \big) }       {|z|^{r-1}} {[ \overline  z]^{ |\um|}} w (z^r) \, i dz \nwedge d \overline z ,
	\end{align*}
	where $u^{1/r}$ is the principal branch of the $r$-th root of $u$. In the polar coordinates, write $z = x e^{i \phi}$ and $u = y e^{i \omega}$, then the integral above turns into
	\begin{align*}
	\frac {2} {y^{(r-1)/ r} e^{i  |\um|  \omega / r} }   \int_0^{2 \pi} \hskip - 2  pt \int_{T^{1/r}}^{(\varDelta T)^{1/r}} \hskip - 8 pt  { e \big(  2 r y^{1/r}   x \cos \lp \phi + \omega/r \rp - 2 \rho x^r \cos (r \phi)    \big) }       {x^{r }} {e^{ - i |\um| \phi}} w (x^r e^{i r \phi})  dx d\phi.
	\end{align*}
	After changing the variable $x$ to $  y^{1/ r (r-1)} x / \rho^{1/(r-1)}    $, we obtain 
	\begin{align}\label{6eq: integral, complex}
	h  (\lambdaup, \omega) = \frac {2 T } {\rho \, \lambdaup^{r-2} e^{i  |\um|  \omega / r} } \int_0^{2 \pi} \int_{ X }^{\varDelta^{1/r} X } e \lp \lambdaup p  (x, \phi; \omega) \rp g (x, \phi, \lambdaup) d x d \phi,
	\end{align}
	with  
	\begin{align*}
	& \lambdaup = \lp \frac y   \rho \rp^{1/{(r-1)}}, \hskip 10 pt X %= \frac {T^{1/r} \rho^{1/(r-1)}} {y^{1/r(r-1)}} 
	= \lp \frac {T \rho} {\lambdaup} \rp^{1/r}, \\
	 p  (x, \phi; \omega) =   \ 2 r   x \cos ( \phi   &  + \omega/r )   - 2  x^r \cos (r \phi)   ,  \hskip 10 pt  g (x, \phi, \lambdaup) =    \frac {\lambdaup x^r {e^{ - i |\um| \phi}} } {T \rho }    w \lp \frac {\lambdaup x^r e^{i r \phi}} {\rho} \rp,
	\end{align*}
	and the weight function $g (x,  \phi, \lambdaup)$ satisfying 
	\begin{align*}
	(\partial / \partial x)^j (\partial / \partial \phi)^k (\partial / \partial \lambdaup)^l g (x, \phi, \lambdaup) \lll X^{-j}\lambdaup^{- l}.
	\end{align*}
	We now apply the method of stationary phase for double integrals; a similar integral  is also treated in \cite{Qi-II-G}  for $r=2$. Since
	\begin{align*}
	p' (x, \phi; \omega) = \lp 2 r \lp \cos ( \phi + \omega / r) - x^{r-1} \cos (r \phi)   \rp,  - 2 r x \lp     \sin ( \phi + \omega / r) - x^{r-1} \sin (r \phi)  \rp \rp,
	\end{align*}
    there is a unique stationary point $(x_0, \phi_0) = (1, \omega/r (r-1))$. 
	First, in the case that either $X < 1/\varDelta$ or $X > \varDelta$, we repeatedly apply  the elaborated partial integration of H\"ormander (see the proof of \cite[Theorem 7.7.1]{Hormander} and also \cite[\S 5.1.3]{Qi-II-G}). More precisely, we define
	\begin{align*}%\label{5eq: g}
	q (x, \phi; \omega) & = x^{1-r} \lp \partial_ x p (x, \phi; \omega) \rp^2 + x^{-1-r} \lp \partial_\phi   p (x, \phi; \omega) \rp^2 \\
	& = 4r^2 \lp x^{r-1}  + x^{1-r} - 2 \cos ((r-1)\phi - \omega/r ) \rp.
	\end{align*}
	An important observation is that  $q (x, \phi; \omega)$ has a uniform positive lower bound,
	\begin{equation*}%\label{5eq: bound for g}
	q (x, \phi; \omega) \geqslant  4r^2 \big( \varDelta^{  (r-1)^2/2r} - \varDelta^{- (r-1)^2/2r} \big)^2, %  \max \left\{ X^{r-1}, X^{1-r} \right\}, 
	\hskip 10 pt  x   \in [X, \varDelta^{1/r} X]. 
	\end{equation*}
	Then, writing the integral in \eqref{6eq: integral, complex} as
	\begin{align*}
	     \frac {1}{2 \pi i \, \lambdaup}  \int_0^{2 \pi} \int_X^{\varDelta^{1/r} X} \frac {g(x, \phi, \lambdaup)} {q (x, \phi; \omega)} \bigg( & \frac {  \partial_{x} p(x, \phi; \omega)} {x^{r-1} } \partial_x (   e \lp \lambdaup p \lp x, \phi; \omega \rp \rp )   \\
	&    +   \frac {   \partial_{\phi} p(x, \phi; \omega)} {x^{r+1}  } \partial_\phi (   e \lp \lambdaup p \lp x, \phi; \omega \rp \rp ) \bigg)  d x d \phi,
	\end{align*}
	and integrating by parts, we obtain
	\begin{align*}
	    - \frac {1}{2 \pi i \, \lambdaup}  \int_0^{2 \pi} \int_X^{\varDelta^{1/r} X} \bigg( \frac {\partial} {\partial x} & \lp \frac {g(x, \phi, \lambdaup) \partial_{x} p(x, \phi; \omega)} {x^{r-1} q (x, \phi; \omega) }      \rp    \\
	   &  +  \frac {\partial} {\partial \phi } \lp  \frac {  g(x, \phi, \lambdaup) \partial_{\phi} p(x, \phi; \omega)} {x^{r+1} q (x, \phi; \omega) } \rp \bigg)    e \lp \lambdaup p \lp x, \phi; \omega \rp \rp     d x d \phi.
	\end{align*}
	By calculating the derivatives in the integrand and  estimating each resulting integral, it is easy to see that we get a saving of $ X \, \lambdaup$. Repeating this elaborated partial integration  $2 K$ times, it may be shown that we would save an $ X \, \lambdaup$ each time and hence
	\begin{align*}
	h  (\lambdaup, \omega) \lll \frac {  T X } {\rho \, \lambdaup^{r-2}   } \frac 1 {(X \, \lambdaup)^{2K}} = \frac { T^{(r+1)/ r} } {\lp \rho \, \lambdaup^{r-1} \rp^{(r-1)/r}   } \frac 1 {(X \, \lambdaup)^{2K}} = \frac {T^2} {(Ty)^{(r-1)/r + 2K/r}}.
	\end{align*}
	Second, when $1/\varDelta \leqslant X \leqslant \varDelta$, we use the stationary phase estimate in \cite[Theorem 1.1.4]{Sogge} (or \cite[Theorem 7.7.5]{Hormander}) to bound the integral in \eqref{6eq: integral, complex} by $\lambdaup^{-1}$, then
	\begin{align*}
	h  (\lambdaup, \omega)  \lll \frac {T } {\rho \, \lambdaup^{r-2 } } \, \frac 1 {\lambdaup } =  \frac {T } {y}.
	\end{align*}
\end{proof}

When $T |\rho| \lll 1$, we are in a much easier situation. Since $e (-\rho x)$ and $e(-\rho z - \overline \rho \overline z)$  are non-oscillatory, they may be absorbed into the weight functions, and we may save any power of $(T |y|)^{1/r}$  or $ (T |u|)^{1/r}$  by repeating partial integration.

\begin{lem} \label{lem: bound for Hankel, 3}
	Let $K \geqslant 0$. Put $C = C_{K,\, \umu}$ or $C_{2 K, (\umu, \um)}$  as in Proposition {\rm\ref{6prop: asymptotic} }. Let $y \in \BR^{\times}$ and $u \in \BCx$.   Suppose that $T |\rho| $ is bounded, say by $  \varDelta^{r-1}C^{1/r} $, and that  $  T |y|, \, T |u| > C  $. Then, 
	\begin{align*}
	\widetilde f (y) \lll_{K,\, \umu,\, \varDelta}    {T }   {\lp T |y| \rp^{-(r-1)/2r - K/ r} },
	\end{align*} 
	if $F = \BR$, and
	\begin{align*}
	\widetilde f (u) \lll_{K,\, (\umu, \um),\, \varDelta}    {T^2 }   {\lp T |u| \rp^{-(r-1)/ r - 2 K/ r} },
	\end{align*} 
	if $F = \BC$.
\end{lem}

In conclusion, combining Lemma \ref{lem: bound for Hankel, 1}, \ref{lem: bound for Hankel, 2} and \ref{lem: bound for Hankel, 3}, we have the following corollary.

\begin{cor}\label{cor: Hankel}
	Let $T, Y > 0 $ be sufficiently large in terms of  $\pi$ and $\varDelta$.  Let $y \in \BR^{\times}$ and $u \in \BCx$.  Then, for any $\rho \in F$ with $   |\rho | \leqslant Y / \varDelta^{r-1} T $, we have uniform bounds
	\begin{align*}
	\widetilde f (y) \lll \left\| \frac T y \right\|^{1/2}, \hskip 10 pt \widetilde f (u) \lll \left\| \frac T u \right\|^{1/2}, 
	\end{align*}
	and, for any $A >  \frac 1 2$, we have
	\begin{align*}
	\widetilde f (y) \lll    \|{T }\|^{1-A}   {\|  y \|^{-A} }, \hskip 10 pt \widetilde f (u) \lll   \|{T }\|^{1-A}   {\|   u \|^{-A} },
	\end{align*}
	if  $T |y|, \, T |u| >  Y^r$. All the implied constants depend only on  $\pi$, $\varDelta$, and in addition on $A$ for the last two inequalities.
\end{cor}

\section{Proof of Theorem \ref{thm: main, smooth}}

Let  $T \ggg 1$, $\theta \in F_{\infty}$, and let $w $ be a smooth product function on $F_{\infty}$ as in Theorem \ref{thm: main, smooth}.
Let $Q = \sqrt T$ and let $\breve \alpha \in \frO,$  $\breve \beta \in \frO \smallsetminus \{0\}$ be chosen as in Lemma \ref{lem: Dirichlet}. By the Chinese remainder theorem, there is $\delta \in \frO \smallsetminus \{0\}$ such that $  \alpha = \breve\alpha/\delta$ and $  \beta = \breve \beta/\delta$  are relatively prime in $\frO_{\vv}$ for all non-archimedean places $\vv$ with $\|\beta \|_{\vv} < \|\alpha \|_{\vv}$.
Let $\rho_\vv =    \alpha  /  \beta - \theta_\vv $. Then, in view of    \eqref{5eq: Dirichlet, 1}, we have 
\begin{align*}
|\breve \beta|_{\vv} \lll Q = \sqrt T,  \hskip 10 pt 
|\rho_{\vv} | \lll \frac 1 {    Q |\breve \beta  |_\vv }  = \frac 1 {  \sqrt T |\breve \beta|_\vv }
\end{align*} 
for all $\vv \in S_{\infty}$. %Note that we have $\sqrt T / |\breve \beta |_\vv \ggg 1$. 
Set $Y_{\vv} = C_{\vv} \sqrt {T } / |  \breve \beta|_{\vv} $, for certain constants $C_{\vv}$ depending only on $F$, $\pi $ and $\varDelta$, which are expressible in terms of those implied constants in Lemma \ref{lem: Dirichlet} and Corollary \ref{cor: Hankel}.

\delete{ Set 
	\begin{align*}
	& f_{\vv} (x) = w_{\vv} (x) e \lp \rho_{\vv} x \rp, \hskip 31 pt \text{ if } F_{\vv} = \BR,\\
	& f_{\vv} (z) = w_{\vv} (z) e \lp   \rho_{\vv} z + \overline \rho_{\vv} \overline z \rp, \hskip 10 pt \text{ if } F_{\vv} = \BC.
	\end{align*} 
	Then,
	\begin{align}\label{7eq: bounds of fv}
	\frac {d^j} {dx^j}  f_{\vv} (x) \lll_j \lp \frac 1 {T } + \frac 1 {|\beta_\vv | Q } \rp^j, \hskip 10 pt \frac {\partial^j} {\partial z^j} \frac {\partial^{j'}} {\partial \overline z^{j'}} f_{\vv} (z) \lll_j \lp \frac 1 {T } + \frac 1 {|\beta_\vv | Q}\rp^{j+ j'}.
	\end{align}
}

Now, let $S_{ \theta, \, w }$ denote the sum in \eqref{0eq: weighted sum, r= 2} or \eqref{0eq: weighted sum, r= 3}, and define $f_{\vv}$ as in \eqref{6eq: choice of f, real} and \eqref{6eq: choice of f, complex}, then
\begin{align*}
S_{ \theta, \, w } = \sum_{ \gamma \, \in \frOO^{\hskip 0.5 pt \prime}     } A_{\pi} (\gamma \mathfrak{D} ) e \lp \Tr \lp \frac   {  \alpha \gamma} {  \beta} \rp \rp f (\gamma),
\end{align*}
which is the left hand side of the \Voronoi summation formula  \eqref{2eq: Classical Voronoi, r=2} or \eqref{2eq: Classical Voronoi, r=3} in Proposition \ref{2prop: classical Voronoi}. In the following, we shall estimate the right hand side.

When $r = 2$, applying Corollary \ref{cor: Hankel} to bound $\widetilde f  \lp \gamma   \rp$ on the right hand side of \eqref{2eq: Classical Voronoi, r=2}, we obtain
\begin{align*}
S_{ \theta, \, w } \lll \frac { T^{N/2}  } { \RN   \frb } \hskip - 3pt \sum_{ \sstyle \gamma \, \in \Fx \cap F_{\infty} (  Y^2/T) \atop \sstyle \gamma \frb^2 \subset \frOO^{\hskip 0.5 pt \prime} } \hskip - 3pt
\frac { |A_{\widetilde \pi} \lp \gamma \frb^2 \mathfrak{D} \rp| } {|\RN \gamma|^{1/2}} + \frac { T^{N/2}  } { \RN   \frb } \hskip - 2pt
\sum_{ \sstyle S \subset S_{\infty} \atop \sstyle S \neq \O    } \hskip - 5pt \frac 1 {\|T\|_S } \sum_{ \sstyle \gamma \, \in \Fx \cap F_{\infty}^{S} (  Y^2/T) \atop \sstyle \gamma \frb^2 \subset \frOO^{\hskip 0.5 pt \prime} }
\frac { |A_{\widetilde \pi} \lp \gamma \frb^2 \mathfrak{D} \rp| } {|\RN \gamma|^{1/2} \|\gamma\|_S },
\end{align*}
where the   ideal $\frb  $ is defined as $  \prod_{ \|\beta \|_{\vv} < \|\alpha \|_{\vv} } \frp_{\vv}^{\mathrm{ord}_{\vv} \, \beta}  $ (see Proposition \ref{2prop: classical Voronoi}). 
%with $C_{\vv}$ some constant depending only on  $\pi_{\vv}$.
Note that $Y_{\vv}^2 / T = C_{\vv}^2 /|\breve \beta|_{\vv}^2 $. Let us assume momentarily that $ \RN (Y^2/T)  \RN ( \frb^2 \mathfrak{D} ) \geqslant 1 $. Then, by Lemma \ref{lem: average over integers} and \ref{lem: average over integers, 2}, along with $\RN \frb  \leqslant |\RN \breve \beta|$ (note that both $\breve \alpha$ and $\breve \beta$ are integral) and $|\breve \beta|_{\vv} \lll \sqrt T$, we find that both the first  and the second sum are bounded by
\begin{align*}
\frac { T^{N/2}  } { \RN   \frb }  \frac { \RN \frb ^{2 + 2\varepsilon } } { |\RN \breve \beta|^{1 + 2\varepsilon}  } \leqslant T^{N/2  },
\end{align*}
and
\begin{align*}
\frac { T^{N/2}  } { \RN   \frb }  \sum_{ \sstyle S \subset S_{\infty} \atop \sstyle S \neq \O    } \frac 1 { \|T\|_S   } \frac {\| \breve \beta\|_S^2  \RN \frb^{2 + 2\varepsilon} } { |\RN\breve \beta |^{1 + 2\varepsilon} } \lll T^{N/2}  .
\end{align*}
When  $ \RN (Y^2/T)  \RN ( \frb^2 \mathfrak{D} ) < 1 $, the argument actually simplifies because the first sum has no terms and the second can   be taken over a smaller range by rescaling $Y$. 

When $r = 3$, we apply similar arguments. We estimate the right hand side of \eqref{2eq: Classical Voronoi, r=3} using Weil's bound \eqref{2eq: Weil's bound} for the Kloosterman sum $S_{\frb \fra'^{-1}} \big( 1, - \widebar \alpha   \beta^3 \gamma/ \gamma'^2  ; \beta / \gamma' \big)$ and Corollary \ref{cor: Hankel} for $\widetilde f  \lp \gamma   \rp$. Then
\begin{align*}
S _{ \theta, \, w } \lll  \sum_{  \frb \, \subset \fra' \, \subset \frOO^{\hskip 0.5 pt \prime}}     \frac {  (\RN \fra' )^{1/2-\varepsilon}} { (\RN \frb)^{3/2- \varepsilon } }   \lp M   (\fra',   \frb; \breve \beta) + E  (\fra',   \frb; \breve \beta) \rp
\end{align*}
with 
\begin{align*}
& M (\fra',   \frb; \breve \beta) = T^{N/2} \sum_{ \sstyle \gamma \, \in \Fx \cap F_{\infty} (  Y^3/T) \atop \sstyle \gamma \frb^3 \fra'^{-2} \subset \frOO^{\hskip 0.5 pt \prime} }  \frac {\left|A_{\widetilde \pi} \lp \fra' \mathfrak{D} , \gamma \frb^3 \fra'^{-2} \mathfrak{D} \rp \right|} {  |\RN \gamma |^{1/2} } , \\
& E (\fra',   \frb; \breve \beta) = T^{N/2} \sum_{ \sstyle S \subset S_{\infty} \atop \sstyle S \neq \O    }  \frac 1 {\|T\|_S }  \sum_{ \sstyle \gamma \, \in \Fx \cap F_{\infty}^{S} (  Y^3/T) \atop \sstyle \gamma \frb^3 \fra'^{-2} \subset \frOO^{\hskip 0.5 pt \prime} }  \frac {\left|A_{\widetilde \pi} \lp \fra' \mathfrak{D} , \gamma \frb^3 \fra'^{-2} \mathfrak{D} \rp \right|} {  |\RN \gamma |^{1/2} \|\gamma\|_S }.
\end{align*}
Note that $Y_{\vv}^3 / T = C_{\vv}^3 \sqrt T /|\breve \beta|_{\vv}^3 $. Again, we shall assume that $\RN (Y^3/T) \RN (  \frb^3   \fra^{-2} \mathfrak{D}) \geqslant 1$; otherwise, the analysis is simpler as it was above when $r = 2$. It follows from Lemma \ref{lem: average over integers} and \ref{lem: average over integers, 2}, along with $\RN \frb  \leqslant |\RN \breve \beta|$ and $|\breve \beta|_{\vv} \lll \sqrt T$, that
\begin{align*}
& M (\fra',   \frb; \breve \beta) \lll T^{N/2} \cdot \frac {(\RN \frb)^{3 + 3 \varepsilon} \, T^{  \lp 1/4 + \varepsilon / 2 \rp N }  } {(\RN \fra')^{1 + 2 \varepsilon} |\RN \breve \beta|^{3/2 +3 \varepsilon}} \lll 
\frac {(\RN \frb)^{3/2  } \, T^{    3 N/4  + \varepsilon }  } {(\RN \fra')^{1 +   \varepsilon}  } ,\\
& E (\fra',   \frb; \breve \beta) \lll T^{N/2} \sum_{ \sstyle S \subset S_{\infty} \atop \sstyle S \neq \O    }  \frac 1 {\|T\|_S } \frac { \|\breve \beta\|_S^3   (\RN \frb)^{3 + 3 \varepsilon} \, T^{  \lp 1/4 + \varepsilon / 2 \rp N }   } {\|T\|_S^{1/2}  (\RN \fra')^{1 + 2 \varepsilon} |\RN \breve \beta|^{3/2 +3 \varepsilon}  } \lll 
\frac {(\RN \frb)^{3/2  } \, T^{    3 N/4  + \varepsilon }  } {(\RN \fra')^{1 +   \varepsilon}  }.
\end{align*}
Combing these, we have
\begin{align*}
S _{ \theta, \, w } \lll    \sum_{  \frb \, \subset \fra' \, \subset \frOO^{\hskip 0.5 pt \prime}}     \frac {  (\RN \fra' )^{1/2-\varepsilon}} { (\RN \frb)^{3/2- \varepsilon } } \frac {(\RN \frb)^{3/2  } \, T^{    3 N/4  + \varepsilon }  } {(\RN \fra')^{1 +   \varepsilon}  } = T^{    3 N/4  + \varepsilon } \sum_{  \frb \, \subset \fra' \, \subset \frOO^{\hskip 0.5 pt \prime}}  \frac {(\RN \frb)^{ \varepsilon }     } {(\RN \fra')^{1/2 +   \varepsilon}  } \lll T^{    3 N/4  + \varepsilon }.
\end{align*}

\section{Proof of Theorem \ref{thm: main}}

In this final section, we show how to deduce Theorem \ref{thm: main} from Theorem \ref{thm: main, smooth} following \cite[\S 8]{Additive-LiY}. %The uniformity in $\theta$ play a critical role in this deduction.
First, we have the following higher dimensional generalization of \cite[Lemma 8.1]{Additive-LiY}.

\begin{lem} \label{8lem: hX}
	Let $\varPi' \subset F_{\infty}$ be a fundamental parallelotope for  $F_{\infty} / \frO'$ which is symmetric about zero.
	For $ X \geqslant 2$ there exists a function $h_{X} (x)$ on $F_{\infty} $ such that 
	\begin{align}\label{8eq: L1 bound}
	\int_{F_{\infty}} |h_X (x) | d x \lll (\log  X)^N,
	\end{align}
	with the implied constant depending only on $F$, and that for $\gamma \in \frO'$,
	\begin{equation}\label{8eq: Fourier}
	\int_{F_{\infty} }  h_X (x)  e(\Tr (x \gamma )) d x = 
	\left\{
	\begin{split}
	& 1, \hskip 10 pt \text{ if } \gamma \in X \cdot \varPi',\\
	& 0, \hskip 10 pt \text{ otherwise.} 
	\end{split}\right.
	\end{equation}
	
\end{lem}

\begin{proof}
	%Let $\lambdaup = \min \left\{ \|\gamma\| : \gamma \in \frO \smallsetminus \{0\} \right\}$. 
	By a linear transform $A$, we may transform $\frO'$ to the lattice $\BZ^n$ and $\varPi'$  to the unit hypercube centered at zero.  Define a function on $\BR$,
	\begin{equation*}
	g_X (t) = \left\{ \begin{split}
	& \min \{1,  X/2 + 1 - |t|   \}, \hskip 10 pt \text{ for } |t| \leqslant X/2+1, \\
	& 0, \hskip 94.5 pt \text{ for } |t| > X/2+1.
	\end{split} \right.
	\end{equation*}
	We form the product $ g_X^N (v) = \prod_{n=1}^N g_X (v_n) $  and let $h_X (x)$ be the Fourier transform (with Fourier kernel $e (- \Tr (x y))$) of $g_X^N (A y)$. 
	
	Note that $g_X^N (A \gamma)  = 1$ if $\gamma \in  X \cdot \varPi'$ and vanishes otherwise, giving \eqref{8eq: Fourier}. 
	
	As for \eqref{8eq: L1 bound}, we first prove that  the $L^1$ norm of the Fourier transform   $\widehat g_X  $ is  $O(\log X)$. The proof is    similar to that of \cite[Lemma 9]{DFI-Bilinear}. Since $|\det A| h_X (I \, {^t \hskip -2pt A} x)$ splits into the product $ 2^{r_2} \prod_{n=1}^N  \widehat g_X (x_n)$, the $L^1$ norm of $h_X (x)$ is of size $O \lp (\log X)^N \rp$. Here $I$ is the diagonal matrix  that represents the bilinear form $\Tr (x y)$ so that $\Tr (xy) = {^t \hskip - 1 pt x} I y$; the diagonal entries of $I$ are either $1$ or $-1$.
\end{proof}

Now we   prove Theorem \ref{thm: main}. Let  $\varDelta \ggg 1 \ggg  \varDelta'  $ be suitable parameters such that $T\cdot \varPi' \smallsetminus (T/2)   \varPi' \subset A [ 2 \varDelta' T, \varDelta \varDelta' T / 2]$, with the definition $A[Y, Z] = \left\{ x \in F_{\infty} : |x_{\vv}| \in [Y, Z]  \right\} $. We choose a weight function $w$ that satisfies the conditions in Theorem \ref{thm: main, smooth} with $T$ replaced by $ \varDelta' T$ and such that $ w (x) \equiv 1 $ on the annulus $ A [ 2 \varDelta' T, \varDelta \varDelta' T / 2]$. Thus, %
\begin{align*}
\sum_{\sstyle \gamma\, \in \frOO^{\hskip 0.5 pt \prime}   \cap \left\{ T \cdot \varPi' \smallsetminus (T/2)   \varPi' \right\}  }    A_{\pi} (\gamma \mathfrak{D}) e \lp \Tr (\theta \gamma) \rp  = \sum_{\sstyle \gamma\, \in \frOO^{\hskip 0.5 pt \prime}  \cap \left\{ T \cdot \varPi' \smallsetminus (T/2)   \varPi'  \right\}  }   A_{\pi} (\gamma \mathfrak{D}) e \lp \Tr (\theta \gamma) \rp w (\gamma).
\end{align*}
In view of \eqref{0eq: S theta (T)}, the left hand side is $S_{\theta} (T) - S_{\theta} (T/2) $. Write the right hand side as $S_{ \theta, w} (T) - S_{ \theta, w} (T/2)$. Then   apply Lemma \ref{8lem: hX} with $X = T, T/2$ to get
\begin{align*}
S_{\theta, w} (X)  \leqslant \int_{F_\infty} |h (x)| \left| \sum_{\gamma \in \frOO^{\hskip 0.5 pt \prime} }  A_{\pi} (\gamma \mathfrak{D}) e \lp \Tr \lp (\theta + x) \gamma \rp \rp w (\gamma) \right| d x.
\end{align*}
Theorem \ref{thm: main, smooth} applies to the sum over $\gamma$, the uniformity in $\theta$ being critical, and \eqref{8eq: L1 bound} controls the integral over $F_{\infty}$, giving Theorem \ref{thm: main}.

\appendix

		\section{Dirichlet Approximation for Number Fields}
		
		In this appendix, we give a generalization of Dirichlet's approximation theorem to an arbitrary number field. 
		
		First, we have the following result of simultaneous Dirichlet approximation for lattices in $\BR^N$, which is a variation of \cite[II Theorem 1E]{Schmidt-DA}.
		
		\begin{prop}\label{3thm: Dirichlet}
			Let $\varLambda \subset \BR^N$ be a lattice and let $\varPi $ be a closed fundamental parallelotope for $ \varLambda $. Suppose that $a_{jk}$, $j = 1, ..., M$, $k = 1, ..., N$,  are $MN$ many real numbers and that $Q > 1$ is a real number. Then there exist integers $q_1, ..., q_M$ and a lattice point $\lambdaup \in \varLambda$ with
			\begin{align*}
			& 1  \leqslant  \max \left\{ |q_1|, ..., |q_M| \right\} < Q^{N/M} , \\
			( \sum_{ j = 1}^M & q_j a_{jk}  - \lambdaup_j )_{k = 1, ..., N} \in 2 Q\- \cdot \varPi.
			\end{align*}
		\end{prop}
		As a consequence, we have the following lemma.
		\begin{lem}\label{lem: Dirichlet}
			Let $\theta \in F_{ \infty}$, and let $Q > 1$ be real. Then there exist $\alpha, \beta \in \frO$, with $\beta \neq 0$, such that
			\begin{align}
			\label{5eq: Dirichlet, 1}   | \beta |_{\vv} \lll Q , \hskip 10 pt |\beta_{\vv}   \theta_{\vv}  - \alpha_{\vv}  |_{\vv}  \lll \frac 1 {Q } ,  
			\end{align} 
			for all $\vv \in S_{\infty}$, where the implied constants depend only on the field $F$.
		\end{lem}
		
		\begin{proof}
			In the notations of Proposition \ref{3thm: Dirichlet}, let $\varLambda =  \frO^\sigma$, for a fixed integer basis $\beta_1, ..., \beta_N$ for $\frO$, let $\varPi = \left\{  \sum_{j=1 }^N t_{j} \beta_{j}^{\, \sigma}   : t_{j} \in \left[- \frac 1 2, \frac 1 2  \right]  \right\}$, and let $a_{j\, \vv} = \theta_{\vv} \, \beta_{j \, \vv}$. Recall that $\sigma  $ is the canonical embedding $F \hookrightarrow F_{\infty} = \BR^N$. Note that $M = N$. Then Proposition \ref{3thm: Dirichlet} implies the existence of integers $q_{\vv}$ and $\alpha \in \frO$ such that 
			\begin{align*}
			& 1  \leqslant  \max \left\{ |q_\vv|  \right\} < Q , \\
			(\theta_\vv \sum_{ j = 1}^N  q_j & \beta_{j\, \vv}  - \alpha_{\vv})_{\vv\,\in S_{\infty}} \in 2 Q\- \cdot \varPi.
			\end{align*}
			Therefore follows the lemma by setting $\beta  = \sum_{ j = 1}^N  q_j  \beta_{j }$.
		\end{proof}

	\bibliographystyle{alphanum}
	\bibliography{references}

\def\cprime{$'$} \def\cprime{$'$}
\begin{thebibliography}{HMQ}

\bibitem[Blo]{Blomer}
V.~Blomer.
\newblock Subconvexity for twisted {$L$}-functions on {${\rm GL}(3)$}.
\newblock {\em Amer. J. Math.}, 134(5):1385--1421, 2012.

\bibitem[DFI]{DFI-Bilinear}
W.~Duke, J.~Friedlander, and H.~Iwaniec.
\newblock Bilinear forms with {K}loosterman fractions.
\newblock {\em Invent. Math.}, 128(1):23--43, 1997.

\bibitem[God]{Additive-Godber}
D.~Godber.
\newblock Additive twists of {F}ourier coefficients of modular forms.
\newblock {\em J. Number Theory}, 133(1):83--104, 2013.

\bibitem[HMQ]{Ho-Mu-Qi}
R.~Holowinsky, R.~Munshi, and Z.~Qi.
\newblock Hybrid subconvexity bounds for {$L(\frac{1}{2},\text{Sym}^2f\otimes
  g)$}.
\newblock {\em Math. Z.}, 283(1-2):555--579, 2016.

\bibitem[H{\"o}r]{Hormander}
L.~H{\"o}rmander.
\newblock {\em The analysis of linear partial differential operators. {I}},
  volume 256 of {\em Grundlehren der Mathematischen Wissenschaften}.
\newblock Springer-Verlag, Berlin, 1983.

\bibitem[IT]{Ichino-Templier}
A.~Ichino and N.~Templier.
\newblock On the {V}orono\u\i\ formula for {${\rm GL}(n)$}.
\newblock {\em Amer. J. Math.}, 135(1):65--101, 2013.

\bibitem[Iwa1]{Iwaniec-Topics}
H.~Iwaniec.
\newblock {\em Topics in classical automorphic forms}, volume~17 of {\em
  Graduate Studies in Mathematics}.
\newblock American Mathematical Society, Providence, RI, 1997.

\bibitem[Iwa2]{Iw-Spectral}
H.~Iwaniec.
\newblock {\em Spectral methods of automorphic forms}, volume~53 of {\em
  Graduate Studies in Mathematics}.
\newblock American Mathematical Society, Providence, RI, {S}econd edition,
  2002.

\bibitem[JS]{J-S-Rankin-Selberg}
H.~Jacquet and J.~A. Shalika.
\newblock On {E}uler products and the classification of automorphic
  representations. {I}.
\newblock {\em Amer. J. Math.}, 103(3):499--558, 1981.

\bibitem[Lan]{Lang-ANT}
S.~Lang.
\newblock {\em Algebraic number theory}, volume 110 of {\em Graduate Texts in
  Mathematics}.
\newblock Springer-Verlag, New York, second edition, 1994.

\bibitem[Li1]{Additive-Li}
X.N. Li.
\newblock Additive twists of {F}ourier coefficients of {${\rm GL}(3)$} {M}aass
  forms.
\newblock {\em Proc. Amer. Math. Soc.}, 142(6):1825--1836, 2014.

\bibitem[Li2]{XLi}
X.Q. Li.
\newblock The central value of the {R}ankin-{S}elberg {$L$}-functions.
\newblock {\em Geom. Funct. Anal.}, 18(5):1660--1695, 2009.

\bibitem[LY]{Additive-LiY}
X.Q. Li and M.~P. Young.
\newblock Additive twists of {F}ourier coefficients of symmetric-square lifts.
\newblock {\em J. Number Theory}, 132(7):1626--1640, 2012.

\bibitem[Mil]{Miller-Wilton}
S.~D. Miller.
\newblock Cancellation in additively twisted sums on {${\rm GL}(n)$}.
\newblock {\em Amer. J. Math.}, 128(3):699--729, 2006.

\bibitem[MS1]{Miller-Schmid-2004-1}
S.~D. Miller and W.~Schmid.
\newblock Summation formulas, from {P}oisson and {V}oronoi to the present.
\newblock In {\em Noncommutative harmonic analysis}, volume 220 of {\em Progr.
  Math.}, pages 419--440. Birkh\"auser Boston, Boston, MA, 2004.

\bibitem[MS2]{Miller-Schmid-2006}
S.~D. Miller and W.~Schmid.
\newblock Automorphic distributions, {$L$}-functions, and {V}oronoi summation
  for {${\rm GL}(3)$}.
\newblock {\em Ann. of Math. (2)}, 164(2):423--488, 2006.

\bibitem[MS3]{Miller-Schmid-2009}
S.~D. Miller and W.~Schmid.
\newblock A general {V}oronoi summation formula for {$GL(n,\Bbb Z)$}.
\newblock In {\em Geometry and analysis. {N}o. 2}, volume~18 of {\em Adv. Lect.
  Math. (ALM)}, pages 173--224. Int. Press, Somerville, MA, 2011.

\bibitem[Qi1]{Qi-II-G}
Z.~Qi.
\newblock On the {F}ourier transform of {B}essel functions over complex
  numbers---{II}: the general case.
\newblock {\em preprint, arXiv:1607.01098}, 2016.

\bibitem[Qi2]{Qi-Bessel}
Z.~Qi.
\newblock Theory of fundamental {B}essel functions of high rank.
\newblock {\em arXiv:1612.03553, to appear in Mem. Amer. Math. Soc.}, 2016.

\bibitem[Sch]{Schmidt-DA}
W.~M. Schmidt.
\newblock {\em Diophantine approximation}, volume 785 of {\em Lecture Notes in
  Mathematics}.
\newblock Springer, Berlin, 1980.

\bibitem[Sog]{Sogge}
C.~D. Sogge.
\newblock {\em Fourier integrals in classical analysis}, volume 105 of {\em
  Cambridge Tracts in Mathematics}.
\newblock Cambridge University Press, Cambridge, 1993.

\end{thebibliography}

\end{document}